\documentclass[reqno,english]{amsart}
\usepackage{latexsym,verbatim,amscd,mathrsfs,array}
\usepackage{amsmath,amssymb,amsthm,amsfonts,graphicx,color}
\usepackage{pdfsync}
\usepackage{epstopdf}
\usepackage[colorlinks=true]{hyperref}
\usepackage{soul}
\newcommand{\inn}{{\quad\hbox{in } }}

\newcommand{\R} {\mathbb R}
\newcommand{\cuad}{{\sqcap\kern-.68em\sqcup}}

\newcommand{\be}{\begin{equation}}
\newcommand{\ee}{\end{equation}}

\newcommand{\s}{{\mathbb S}}

\newtheorem{lem}{Lemma}[section]
\newtheorem{prop}[lem]{Proposition}
\newtheorem{thm}[lem]{Theorem}

\newtheorem{defi}[lem]{Definition}

\newtheorem{rem}[lem]{Remark}
\numberwithin{equation}{section}

\def\r{\mathbb R}

\def\s{\mathbb S}
\def\z{\mathbb Z}

\def\c{\mathbb C}

\DeclareMathOperator{\divergence}{div}
\DeclareMathOperator{\sech}{sech}
\numberwithin{equation}{section}

\begin{document}

\title[Delaunay-type singular solutions for the fractional Yamabe problem]
{Delaunay-type singular solutions for the fractional Yamabe Problem}

%%%%%%%%%%%%%%%%%%%%%%%%%%%%%%%%%%%%%%%%%%%%%%%%%%%%%%%%%%%%%%%%%%%%%%

\author[A. DelaTorre]{Azahara DelaTorre}

\address{Azahara DelaTorre
\hfill\break\indent
Universitat Polit\`ecnica de Catalunya,
\hfill\break\indent
ETSEIB-MA1, Av. Diagonal 647, 08028 Barcelona, Spain}
\email{azahara.de.la.torre@upc.edu}

\author[M. del Pino]{Manuel del Pino}

\address{Manuel del Pino
\hfill\break\indent
Universidad de Chile,
\hfill\break\indent Departamento de
Ingenier\'{\i}a  Matem\'atica and Centro de Modelamiento
 Matem\'atico
\hfill\break\indent (UMI 2807 CNRS), Casilla 170 Correo 3, Santiago,
Chile}
\email{delpino@dim.uchile.cl}

\author[M.d.M. Gonz\'alez]{Mar\'ia del Mar Gonz\'alez}

\address{Mar\'ia del Mar Gonz\'alez
\hfill\break\indent
Universitat Polit\`ecnica de Catalunya
\hfill\break\indent
 ETSEIB-MA1, Av. Diagonal 647, 08028 Barcelona, Spain}
\email{mar.gonzalez@upc.edu}

\author[J. Wei]{Juncheng Wei}
\address{Juncheng. Wei
\hfill\break\indent
University of British Columbia
\hfill\break\indent
 Department of Mathematics, Vancouver, BC V6T1Z2, Canada} \email{jcwei@math.ubc.ca}

\thanks{A. DelaTorre is Supported by MINECO grants MTM2011-27739-C04-01, MTM2014-52402-C3-1-P and  FPI-2012 fellowship, and is part of the Catalan research group 2014SGR1083. M.d.M. Gonz\'alez is supported by MINECO grants MTM2011-27739-C04-01 and MTM2014-52402-C3-1-P, and is part of the Barcelona Graduate School of Math and the Catalan research group 2014SGR1083. M. del Pino has been supported by grants Fondecyt 1150066, Fondo Basal CMM and by Nucleo Mienio CAPDE, NC130017. J. Wei is partially supported by NSERC of Canada.}

\begin{abstract}

We construct Delaunay-type solutions for the fractional Yamabe problem with an isolated singularity
$$ (-\Delta)^\gamma w= c_{n, {\gamma}}w^{\frac{n+2\gamma}{n-2\gamma}}, w>0 \ \mbox{in} \ \R^n \backslash \{0\}.
$$
We follow a variational approach, in which the key is the
computation of the fractional Laplacian in polar coordinates.
\end{abstract}

\date{}\maketitle
%\keywords{D}

%\abbreviations{}

\section{Introduction and statement of the main result}

Fix $\gamma\in(0,1)$ and $n\geq 2+2\gamma$. We consider the problem of finding radial solutions for the fractional Yamabe problem in $\r^n$ with an isolated singularity at the origin. That means that we look for positive, radially symmetric solutions of
\begin{equation}\label{equation0}(-\Delta)^{\gamma}w=c_{n,\gamma}w^{\frac{n+2\gamma}{n-2\gamma}}\text{ in }\r^n \setminus\ \{0\},\end{equation}
where $c_{n,\gamma}$ is any positive constant that, without loss of generality, will be normalized as
 \begin{equation}\label{cte}
c_{n,\gamma}=2^{2\gamma}\left(\frac{\Gamma(\frac{1}{2}(\frac{n}{2}+\gamma))}
{\Gamma(\frac{1}{2}(\frac{n}{2}-\gamma))}\right)^2>0.\end{equation}
 In geometric terms, given the Euclidean metric $|dx|^2$ on $\mathbb R^n$, we are looking for a conformal metric
 \begin{equation*}g_w=w^{\frac{4}{n-2\gamma}}|dx|^2,\ w>0,\end{equation*} with positive constant fractional curvature $Q^{g_w}_\gamma\equiv c_{n,\gamma}$, that is radially symmetric and has a prescribed singularity at the origin. It is a classical calculation that $w_1(r)=r^{-\frac{n-2\gamma}{2}}$ is an explicit solution for \eqref{equation0} with the normalization constant \eqref{cte} (see, for instance, Proposition $2.4$ in \cite{xaviros}).

 Because of the well known extension theorem for the fractional Laplacian $(-\Delta)^{\gamma}$ \cite{CaffarelliSilvestre,CaseChang,MarChang},  we have that equation \eqref{equation0} for the case $\gamma\in(0,1)$ is equivalent to the boundary reaction problem
\begin{equation}\label{equation1}\left\{
\begin{split}
-\divergence(y^{1-2\gamma}\nabla W)=0&\text{ in } \r^{n+1}_+,\\
W=w&\text{ on }\r^n\setminus\{0\},\\
- d_{\gamma}\lim_{y\rightarrow 0}y^{1-2\gamma}\partial_yW=c_{n,\gamma}w^{\frac{n+2\gamma}{n-2\gamma}}&\text{ on }\r^n\setminus\{0\},
\end{split}\right.
\end{equation}
for the constant $d_{\gamma}=\tfrac{2^{2\gamma-1}\Gamma(\gamma)}{\gamma\Gamma(-\gamma)}$.

In a recent paper \cite{CaffarelliJinSireXiong} Caffarelli, Jin, Sire and Xiong characterize all the nonnegative solutions to \eqref{equation1}.
Indeed, let $W$ be any nonnegative solution of \eqref{equation1} in $\r^{n+1}_+$ and suppose that the origin %$\{0\}$
 is not a removable singularity. Then, writing $r=|x|$ for the radial variable in $\mathbb R^n$, we must have that
  $$W(x,y)=W(r,y)\text{ and }\partial_r W(r,y)<0\quad \forall\ 0<r<\infty.$$
In addition, they also provide their asymptotic behavior. More precisely, if $w=W(\cdot,0)$ denotes the trace of $W$, then near the origin
\begin{equation}\label{asymptotics}
c_1r^{-\tfrac{n-2\gamma}{2}}\leq w(x)\leq c_2r^{-\tfrac{n-2\gamma}{2}},
\end{equation}
where $c_1$, $c_2$ are positive constants.

 We remark that if the singularity at the origin is removable, all the entire solutions to \eqref{equation1} have been completely classified (\cite{JinLiXiong,ChenLiOu}) and, in particular, they must be the standard ``bubbles"
\begin{equation*}
w(x)=c\left(\frac{\lambda}{\lambda^2+|x-x_0|^2}\right)^{\frac{n-2\gamma}{2}},\quad c,\lambda>0, \ x_0\in\mathbb R^n.
\end{equation*}

In this paper we study the existence of ``Delaunay''-type solutions for \eqref{equation0}, i.e, solutions  of the form
\begin{equation}\label{wv}
w(r)=r^{-\frac{n-2\gamma}{2}}v(r)\text{ on } \r^n\setminus\{0\},
\end{equation}
 for some function $0<c_1\leq v\leq c_2$ that, after the Emden-Fowler change of variable $r=e^t$, is periodic in the variable $t$. With some abuse of the notation, we write $v=v(t)$.

 In the classical case $\gamma=1$, equation \eqref{equation0} reduces to a standard second order ODE. However, in the fractional case \eqref{equation0} becomes a fractional order ODE, so classical methods cannot be directly applied here.
 Instead, we reformulate the problem into a variational one for the the periodic function $v$.
The main difficulty is to compute the fractional Laplacian in polar coordinates. %

Our approach does not use the extension problem \eqref{equation1}. Instead we work directly with the nonlocal operator, after suitable Emden-Fowler transformation.  For $\gamma\in(0,1)$ we know that the fractional Laplacian can be defined as a singular kernel as
 \begin{equation*}
 (-\Delta)^{\gamma}w(x)=\kappa_{n,\gamma}\text{P.V. }\int_{\r^n}\frac{w(x)-w(x+y)}{|y|^{n+2 \gamma}}\,dy,
\end{equation*}
where $P.V. $ denotes the principal value, and the constant $\kappa_{n,\gamma}$ (see \cite{Landkof}) is given by
$$\kappa_{n,\gamma}=\pi^{-\tfrac{n}{2}}2^{2\gamma}
\tfrac{\Gamma\left(\tfrac{n}{2}+\gamma\right)}{\Gamma(1-\gamma)}\gamma.$$
The main idea here is to use the Emden-Fowler change of variable in the singular integral. After some more changes of variable, equation \eqref{equation0} will be written as
\begin{equation}\label{equation3}
 \mathscr{L}_{\gamma}v=c_{n,\gamma}v^{\beta}, v>0,
\end{equation}
where $$\beta=\tfrac{n+2\gamma}{n-2\gamma}$$ is the critical exponent in dimension $n$ and $ \mathscr{L}_{\gamma}$ is the linear operator defined by
\begin{equation*}
 \mathscr{L}_{\gamma}v(t)=\kappa_{n,\gamma}P.V.   \int_{-\infty}^\infty
 (v(t)-v(\tau))K(t-\tau)\,d\tau + c_{n,\gamma} v(t),
\end{equation*}
for
$K$ a singular kernel which is precisely written in \eqref{K}. The behaviour of $K$ near the origin is the same as the kernel of the fractional Laplacian $(-\triangle)^{\gamma}$ in $\r$ and near infinity it presents an exponential decay.  This kind of kernels corresponds to tempered stable process and they have been studied
in \cite{Kassmann352} and \cite{silvestre}, for instance.

If we take into account just periodic functions $v(t+L)=v(t)$, the operator $\mathscr{L}_{\gamma}$ can be rewritten as
\be\label{LgammaL}
 \mathscr{L}_{\gamma}^Lv(t)=\kappa_{n,\gamma}P.V.   \int_{0}^L
 (v(t)-v(\tau))K_L(t-\tau)\,d\tau + c_{n,\gamma} v(t),
\ee
where $K_L$ is a periodic singular kernel that will be defined in \eqref{KL}.
 For periodic solutions, problem \eqref{equation3} is equivalent to finding a minimizer for the functional
\begin{equation*}
\mathscr{F}_L(v) =\frac{ \kappa_{n,\gamma}\int_0^L \int_0^L (v(t)-v(\tau))^2 K_L (t-\tau) \,dt \,d\tau +c_{n,\gamma} \int_0^L v(t)^2 dt}{ (\int_0^L v(t)^{{\beta}+1} dt)^{\frac{2}{{\beta}+1}}}.
\end{equation*}
Note that a minimizer always exists as we can check in Lemma \ref{Lem1}. The minimum value for the functional will be denoted by $c(L)$. \\

Our main result is the following:
 \begin{thm}\label{th1}
There is a unique $L_0>0$ such that $c(L)$ is attained by a nonconstant minimizer when $L>L_0$ and when $L \leq L_0$, $ c(L)$ is attained by the constant only.
\end{thm}

 In a recent paper \cite{paper1} the authors study this fractional problem \eqref{equation0} from  two different points of view. They carry out an ODE-type study and explain the geometrical interpretation of the problem. In addition, they give some results towards the description of some kind of generalized phase portrait. For instance, they prove the existence of periodic radial solutions for the linearized equation around the equilibrium $v_1\equiv 1$, with period $L_0=L_0(\gamma)$. For the original non-linear problem  they show the existence of a Hamiltonian quantity conserved along trajectories, which suggests that the non-linear problem has periodic solutions too, for every period larger than this minimal period $L_0$. Theorem \ref{th1} proves this conjecture.

The construction of Delaunay solutions allows for many further studies. For instance, as a consequence of our construction one obtains the non-uniqueness of the solutions for the fractional Yamabe equation (see \cite{MarQing} for an introduction to this problem) in the positive curvature case, since it gives different conformal metrics on $\mathbb S^1(L)\times \mathbb S^{n-1}$ that have constant fractional curvature.
 This is well known in the scalar curvature case $\gamma=1$ (see the lecture notes \cite{Schoen:notas} for an excellent review, or the paper \cite{Schoen:number}). In addition, this fact gives some examples for the calculation of the total fractional scalar curvature functional, which maximizes the standard fractional Yamabe quotient across conformal classes.

From another point of view, Delaunay solutions can be used in gluing problems. Classical references are, for instance, \cite{Mazzeo-Pacard:isolated, Mazzeo-Pollack-Uhlenbeck} for the scalar curvature, and \cite{Mazzeo-Pacard:Delaunay-ends,Mazzeo-Pacard-Pollack} for the construction of constant mean curvature surfaces with Delaunay ends. In the non-local case, we use Delaunay-type singularities to deal with the problem of constructing metrics of constant fractional curvature with prescribed isolated singularities \cite{Ao-DelaTorre-Gonzalez-Wei}.

Recently it has been introduced a related notion of nonlocal mean curvature $H_\gamma$ for the boundary of a set in $\mathbb R^n$ (see \cite{Caffarelli-Roquejoffre-Savin,Valdinoci:review}).  Finding Delaunay-type surfaces with constant nonlocal mean
curvature  has  just  been  accomplished  in  \cite{Cabre-Fall-Sola-Weth}.  For  related  nonlocal  equations with periodic solutions see also \cite{Davila-delPino-Dipierro-Valdinoci,CabreMasJoan}.\\

The paper will be structured as follows: in Section $2$ we will introduce the problem. In particular we will recall some known results for the classical case and we will present the formulation of the problem through some properties of the singular kernel. In Section $3$ we will show some technical results that we will need in the last Section; where we will use the variational method to prove the main result in this paper, this is, Theorem \ref{th1}.

\section{Set up of the problem}

\subsection{Classical case $\gamma=1$}

We recall first some well known results. Consider the classical Yamabe problem
\be\label{eqclasica}
(-\Delta) w = c_{n, 1}w^{\frac{n+2}{n-2}},\  w >0  \inn \R^n \backslash \{0\}.
\ee
Because of the asymptotic behavior in \eqref{asymptotics} we look at radially symmetric solutions of the form
\begin{equation*}
w(r)= r^{-\frac{n-2}{2}} v (r),
\end{equation*}
where $r=|x|.$
Then, applying the Emden-Fowler change of variable $e^t=r$, equation \eqref{eqclasica}
 reads

\begin{equation*}
 -\ddot{v}+\tfrac{(n-2)^2}{4}v=\tfrac{(n-2)^2}{4}v^{\frac{n+2}{n-2}},\quad v>0.
 \end{equation*}
 This equation is easily integrated and the analysis of its phase portrait gives that all bounded solutions must be periodic (see, for instance, the lecture notes \cite{Schoen:notas}). More precisely, the Hamiltonian
 \begin{equation*}H_1(v,\dot{v}):= \tfrac{1}{2}\dot v^2+\tfrac{(n-2)^2}{4}\left(\tfrac{(n-2)}{2n} v^{\frac{2n}{n-2}}-\tfrac{1}{2}v^2\right)\end{equation*}
 is preserved along trajectories. Thus, looking at its level sets we conclude that there exists a family of periodic solutions $\{v_L\}$ of periods $L\in(L_0,\infty)$. Here
 \begin{equation*}
  L_0=\tfrac{2\pi}{\sqrt{n-2}}
 \end{equation*}
 is the minimal period and it can be calculated from the linearization at the equilibrium solution $v_1\equiv 1$. These $\{v_L\}$ are known as the Fowler or Delaunay solutions for the scalar curvature.

\subsection{Formulation of the problem.} We now consider the singular Yamabe problem

\be
(-\Delta)^{\gamma} w = c_{n, {\gamma}}w^{\beta} \inn \R^n \backslash \{0\}, \ w >0
\label{1}\ee
for $ \gamma \in (0,1)$, $n\geq 2$, $\beta$ the critical exponent given by
\begin{equation*}
\beta=\frac{n+2\gamma}{n-2\gamma}
\end{equation*}
and
\begin{equation*}\label{deffraclapla}
 (-\Delta)^{\gamma}w(x)=\kappa_{n,\gamma}\text{P.V. }\int_{\r^n}\frac{w(x)-w(x+y)}{|y|^{n+2 \gamma}}\,dy,
\end{equation*}
where $P.V. $ denotes the principal value, and the constant $\kappa_{n,\gamma}$ (see \cite{Landkof}) is given by
$$\kappa_{n,\gamma}=\pi^{-\tfrac{n}{2}}2^{2\gamma}
\tfrac{\Gamma\left(\tfrac{n}{2}+\gamma\right)}{\Gamma(1-\gamma)}\gamma.$$
Because of \eqref{asymptotics} we only consider radially symmetric solutions of the form
\begin{equation*}
w(x)= |x|^{-\frac{n-2\gamma}{2}} v (|x|),
\end{equation*}
where $v$ is some function $0<c_1\leq v\leq c_2$.
In radial coordinates $(r=|x|, \theta\in\s^{n-1}$ and $s=|y|, \sigma\in\s^{n-1}$), we can express the fractional Laplacian as
\begin{equation*}
(-\Delta )^{\gamma} u = \kappa_{n, {\gamma}} P.V.  \int_{0}^\infty \int_{\s^{n-1}} \frac{ r^{-\frac{n-2\gamma}{2}} v(r)-s^{-\frac{n-2\gamma}{2}} v(s)}{|r^2+s^2-2 rs \langle\theta , \sigma\rangle|^{\frac{n+2\gamma}{2}}} s^{n-1}  d\sigma ds.
\end{equation*}
Inspired in the computations by Ferrari and Verbitsky in \cite{FerrariVerbitsky}, we write
 \begin{equation*}
 s= r \bar{s},
 \end{equation*}
so the radial function $v$ can be expressed as
 $$v(r)=  (1-\bar{s}^{-\frac{n-2\gamma}{2}}) v(r)+ \bar{s}^{-\frac{n-2\gamma}{2}} v(r).$$
 Thus the equation \eqref{1} for $v$ becomes
\be\label{eqvt}
\kappa_{n,\gamma}P.V.   \int_{0}^\infty \int_{\s^{n-1}} \frac{ \bar{s}^{ n-1-\frac{n-2\gamma}{2}} ( v(r)- v(r \bar{s}))}{|1+\bar{s}^2-2 \bar{s} \langle\theta , \sigma\rangle|^{\frac{n+2\gamma}{2}}}  d\sigma d\bar{s} + A v= c_{n,\gamma}v^{\beta}(r),
\ee
where
\begin{equation*}
A= \kappa_{n,\gamma}P.V.   \int_{0}^\infty \int_{\s^{n-1}} \frac{ (1-\bar{s}^{-\frac{n-2\gamma}{2}}) \bar{s}^{n-1}}{|1+\bar{s}^2-2 \bar{s} \langle\theta , \sigma\rangle|^{\frac{n+2\gamma}{2}}}  d\sigma d\bar{s}.
\end{equation*}
\begin{rem}
The constant $A$ is strictly positive. Indeed, from \eqref{eqvt} we have
\begin{equation*}
A=c_{n,\gamma}>0,
\end{equation*}
since $c_{n,\gamma}$ is normalized such that $v_1\equiv 1$ is a solution for the singular Yamabe problem (see Proposition $2.7$ in \cite{paper1}).
\end{rem}
Finally we do the Emden-Fowler changes of variable $r=e^t$ and  $s=e^{\tau}$  in \eqref{eqvt} to obtain
\be
    \mathscr{L}_{\gamma}v= c_{n,\gamma}v^{\beta},
    \label{10}
\ee
where the operator $\mathscr{L}_{\gamma}$ is defined as
\be\label{L}
\mathscr{L}_{\gamma}v=\kappa_{n,\gamma}P.V.   \int_{-\infty}^\infty (v(t)-v(\tau))K(t-\tau)\,d\tau + c_{n,\gamma} v,
\ee
for a function $v=v(t)$ and the kernel $K$   is given by

\begin{equation}\label{Kk}
K (\xi)= 2^{-\frac{n+2\gamma}{2}}\int_{\s^{n-1}} \frac{1}{ |\cosh (\xi) -\langle\theta , \sigma\rangle|^{\frac{n+2\gamma}{2}} } d \sigma=\int_{\s^{n-1}} \frac{e^{-\frac{n+2\gamma}{2} {\xi}}}{ (1+ e^{-2{\xi}} -2 e^{{-\xi}}\langle\theta , \sigma\rangle )^{\frac{n+2\gamma}{2}}}\,d\sigma.
\end{equation}

\begin{rem}\label{rotinv}
 $K$ is rotationally invariant in the variable $\theta$, thus we drop the dependence on $\theta$ in the argument of $K$. Indeed if we identify $e_1=(1,0,\dots,0)$ with a fixed point in $\s^{n-1}$ via the usual embedding $\s^{n-1}\hookrightarrow \r^n$ and we define $$J(\theta):=\int_{\s^{n-1}}\frac{e^{-\frac{n+2\gamma}{2} {\xi}}}{ |1+ e^{-2{\xi}} -2 e^{-{\xi}}\langle\theta , \sigma\rangle |^{\frac{n+2\gamma}{2}}}\,d\sigma,$$  it is easy to check that  $J(\theta)=J(e_1)$. The proof is trivial using  equality \eqref{Kk}  and the change of variable $\tilde{\sigma}=R^{\top}\sigma$, where $R$ is any rotation such that $R(e_1)=\theta$.
 \end{rem}
 The kernel also can be written using spherical coordinates as
\be\begin{split}\label{K}
K({\xi})
&= \bar c_n e^{-\frac{n+2\gamma}{2} {\xi}} \int_0^\pi \frac{(\sin \phi_1)^{n-2}}{ (1+ e^{-2{\xi}} -2 e^{{-\xi}} \cos \phi_1)^{\frac{n+2\gamma}{2}}} \,d\phi_1\\
&=\bar c_n 2^{-\frac{n+2\gamma}{2}}\int_0^\pi \frac{(\sin \phi_1)^{n-2}}{ (\cosh (\xi) -\cos(\phi_1))^{\frac{n+2\gamma}{2}} } \,d \phi_1,
\end{split}
\ee
where $\phi_1$ is the angle between $\theta$ and $\sigma$, and $\bar c_n$ is a positive dimensional constant that only depends on the integral in the rest of the spherical coordinates.

\begin{rem}
 The expression \eqref{K} implies that $K(\xi)$ is an even function.
 Moreover, since $\phi_1\in(0,\pi)$ and $\cosh(x)\geq 1$, $\forall x\in\R$, $K$ is strictly positive.
\end{rem}

In the next paragraphs we will find a more explicit formula for $K$ that will help us calculate its asymptotic behavior.

\begin{lem}\label{propiedadeshiperg}
 \textnormal{ \cite{Abramowitz,SlavyanovWolfganglay,MOS}} Let $z\in\c$. The hypergeometric function is defined for $|z| < 1$ by the power series
$$ \mbox{}_2F_{1} (\tilde{a},\tilde{b};\tilde{c};{z}) = %\sum_{n=0}^\infty \frac{(\tilde{a})_n (\tilde{b})_n}{(\tilde{c})_n} \frac{z^n}{n!}=
\frac{\Gamma(\tilde{c})}{\Gamma(\tilde{a})\Gamma(\tilde{b})}\sum_{n=0}^\infty \frac{\Gamma(\tilde{a}+n)\Gamma(\tilde{b}+n)}{\Gamma(\tilde{c}+n)} \frac{z^n}{n!}.$$
It is undefined (or infinite) if $\tilde{c}$ equals a non-positive integer.
Some properties of this function are
\begin{enumerate}
  \item The hypergeometric function evaluated at $z=0$ satisfies
  \begin{equation}\label{prop2}
  \mbox{}_2F_{1} (\tilde{a}+j,\tilde{b}-j;\tilde c;0)=1; \  j=\pm1,\pm2,...
  \end{equation}

 \item If $\tilde{a}+\tilde{b}<\tilde{c}$, the following expansion holds
\be\label{prop8}
\mbox{ }_2 F_1 (\tilde{a}, \tilde{b}; \tilde{c}; x)= O(1),\text{ for }z\sim 1.
\ee

\item  If Re $\tilde{b}>0,|z|<1$,
\begin{equation}\label{prop6}
\begin{split}
\mbox{}_2F_{1} (\tilde{a}, \tilde{a}-\tilde{b}+\tfrac{1}{2}; \tilde{b}+\tfrac{1}{2}; z^2)=\frac{\Gamma(\tilde{b}+\frac{1}{2})}{\sqrt{\pi}\Gamma(\tilde{b})}\int_0^{\pi}\frac{(\sin t)^{2\tilde{b}-1}}{(1+2z\cos t+z^2)^{\tilde{a}}}\,dt.
\end{split}
\end{equation}
\item
If $\tilde{a}-\tilde{b}+1=\tilde{c}$, the following identity holds
\be\label{prop7}
\mbox{ }_2 F_1 (\tilde{a}, \tilde{b}; \tilde{a}-\tilde{b}+1; x)= (1-x)^{1-2\tilde{b}} (1+x)^{2\tilde{b}-\tilde{a}-1} \mbox{ }_2 F_1 \left(\tfrac{\tilde{a}+1}{2}-\tilde{b}, \tfrac{\tilde{a}}{2} -\tilde{b} +1; \tilde{a}-\tilde{b} +1; \tfrac{4 x}{(x+1)^2}\right).
\ee
  \item The derivative of the hypergeometric function with respect to the last argument is
  \begin{equation}\label{prop1}
      \frac{d}{dz}\mbox{ }_2 F_1(\tilde a,\tilde b;\tilde  c;z)=\frac{\tilde a\tilde b}{\tilde c}\mbox{ }_2 F_1(\tilde a+1,\tilde b+1;\tilde c+1;z).
    \end{equation}

\end{enumerate}
\end{lem}

\begin{lem}
The kernel $K$ can be expressed in terms of a hypergeometric function as
\be
K({\xi})= c_n(\sinh {\xi})^{-1-2\gamma} (\cosh {\xi})^{ \tfrac{2-n+2\gamma}{2}} \mbox{ }_2 F_1 \left(\tfrac{a+1}{2}-b, \tfrac{a}{2} -b +1; a-b +1; (\textnormal{sech } {\xi})^2\right),
\label{12}
\ee
where $c_n=\bar c_n 2^{-\frac{n+2\gamma}{2}}\frac{\sqrt{\pi}\Gamma(\frac{n-1}{2})}{\Gamma(\frac{n}{2})}$, %
and $\mbox{ }_2 F_1$ is  the hypergeometric function defined in Lemma \ref{propiedadeshiperg}.
\end{lem}

\begin{proof}
Because of the parity of the kernel $K$ it is enough to study its behavior for $\xi>0$.
Using property \eqref{prop6} given in Lemma \ref{propiedadeshiperg}, with $t=\phi_1$, $\tilde{b}=\frac{n-1}{2}$, $\tilde{a}=\frac{n+2\gamma}{2}$ and $z=-e^{-{\xi}}$, we can assert that, if ${\xi}>0$,
\begin{equation*}
K({\xi})=  \bar c_n\frac{\sqrt{\pi}\Gamma(\frac{n-1}{2})}{\Gamma(\frac{n}{2})}e^{-\frac{n+2\gamma}{2} {\xi}} \mbox{ }_2
F_{1} (a, b; c; e^{-2{\xi}}),
\end{equation*}
where
\be\label{abc}
a= \tfrac{n+2\gamma}{2},\quad b =1 +\gamma,\quad c=\tfrac{n}{2}.
\ee
An important observation is  that
$$a-b+1=c,$$
which, from property \eqref{prop7} in Lemma \ref{propiedadeshiperg}, %following identity
 yields \eqref{12}.
\end{proof}

\begin{lem}\label{Kexpansion}
The asymptotic expansion of the kernel $K$ is given by
 \begin{itemize}
 \item $K({\xi})  \sim  |{\xi}|^{-1-2\gamma}$ if $|\xi|\rightarrow 0$,
 \item $K(\xi)\sim e^{-|\xi|\tfrac{n+2\gamma}{2}}$ if $|\xi|\rightarrow \infty$.
 \end{itemize}
\end{lem}
\begin{proof}
Note that $K$ is an even function. Using property \eqref{prop8} to estimate expression \eqref{12} for $K({\xi})$, we obtain that, for $|{\xi}|$ small enough,
\be\label{Kexp}
K({\xi}) \sim  |\sinh {\xi}|^{-1-2\gamma}  \sim  |{\xi}|^{-1-2\gamma}.
\ee

Moreover, this expression \eqref{12}, the behaviour of the hyperbolic secant function  at infinity %
and  the hypergeometric function property \eqref{prop2} given in Lemma \ref{propiedadeshiperg}
show the exponential decay of the kernel at infinity:
\begin{equation}\label{Kinfty}
%\begin{split}
K({\xi})\sim c_n(\sinh {\xi})^{-1-2\gamma} (\cosh {\xi})^{ \frac{2-n+2\gamma}{2}}\sim c e^{-|\xi|\tfrac{n+2\gamma}{2}}.
%\end{split}
\end{equation}
where $c$ is a positive constant. %
\end{proof}
\begin{rem}
 The asymptotic behaviour of this kernel near the origin and near infinity given in Lemma \ref{Kexpansion} correspond to a tempered stable process.
\end{rem}

\subsection{Periodic solutions}

We are looking for periodic solutions of (\ref{10}). Assume that $ v(t+L)= v(t)$: in this case equation (\ref{10}) becomes
\begin{equation*}
\mathscr{L}^L_{\gamma}v=\kappa_{n,\gamma}P.V. \int_0^L (v(t)- v(\tau)) K_L (t-\tau) d \tau + c_{n,\gamma} v = c_{n,\gamma}v^{\beta}, \ \ \  \mbox{where}\  \beta=\frac{n+2\gamma}{n-2\gamma},
\end{equation*}
and
\be\label{KL}
K_L (t-\tau)= \sum_{j\in\z}K(t-\tau- jL),
\ee
for $K$ the kernel given in \eqref{K}.
Note that the argument in the integral above has a finite number of poles, but $K_L$ is still well defined.

\begin{lem}\label{appendix}
The periodic kernel $K_L$ satisfies the following inequality:
\be
\frac{L}{L_1} K_{L} \left(\frac{L}{L_1} (t-{\tau})\right) <K_{L_1} (t-{\tau}),\ \ \forall L>L_1>0.
\label{23}
\ee
 \end{lem}
 \begin{proof}
By evenness we just need to show that the function $ {\xi} K ({\xi})$ is decreasing for ${\xi}>0$. By \eqref{12}, up to positive constant,
\be\label{xiKxi}
\begin{split}
 {\xi} K({\xi})= &{\xi} (\sinh {\xi})^{-1-2\gamma} (\cosh {\xi})^{ \frac{2-n+2\gamma}{2}}
 \\&\cdot\mbox{}_2F_{1}  (\tfrac{a+1}{2}-b, \tfrac{a}{2} -b +1; a-b +1; (\sech {\xi})^2),
 \end{split}
\ee
where $a,b,c$ are given in \eqref{abc}.

Observe that
\be\label{hiper}
\mbox{}_2F_{1} (\tfrac{a+1}{2}-b, \tfrac{a}{2} -b +1; a-b +1; (\sech {\xi})^2) >0,
 \ee
 and
$$\tfrac{a+1}{2}-b>0,\quad
 %\tfrac{a+1}{2}-b+1>0,\quad
%\tfrac{a}{2} -b +2>0,\quad
a-b+1=c>0,$$
since $n>2+2\gamma$.
Property \eqref{prop1} yields  that \eqref{hiper} is decreasing. Indeed,
\begin{equation*}
\begin{split}
 \frac{d}{d{\xi}}  \left[\mbox{}_2F_{1}(\tfrac{a+1}{2}-b,\right.&\left. \tfrac{a}{2} -b +1; a-b +1; (\sech {\xi})^2)\right]\\
= &-2\tfrac{ (\tfrac{a+1}{2}-b)( \tfrac{a}{2}-b+1)}{ c}(\sech{\xi})^2 \tanh {\xi}  \\&\cdot\mbox{}_2F_{1}(\tfrac{a+1}{2}-b+1, \tfrac{a}{2} -b +2; a-b +2; (\sech {\xi})^2)<0.
\end{split}
\end{equation*}
Thus we just need to show that the function ${\xi} (\sinh {\xi})^{-1-2\gamma} (\cosh {\xi})^{ \frac{2-n+2\gamma}{2}}$ in \eqref{xiKxi} is decreasing in ${\xi}$. In fact by writing
$${\xi} (\sinh {\xi})^{-1-2\gamma} (\cosh {\xi})^{ \frac{2-n+2\gamma}{2}}= \tfrac{{\xi}}{\sinh {\xi}} (\tanh {\xi})^{-\gamma} (\sinh {\xi})^{-\gamma} (\cosh {\xi})^{ \frac{2-n}{2}}, $$
we have that $ {\xi} K({\xi})$ is a product of positive decreasing functions.

Finally, inequality \eqref{23} follows from the definition of $K_L(\xi)$ given in \eqref{KL}:
\begin{equation*}
\frac{L}{L_1}K_L\left(\frac{L}{L_1}(t-\tau)\right)=\sum_{j=-\infty}^{+\infty}\frac{L}{L_1}K\left(\frac{L}{L_1}(t-\tau-jL_1)\right)
<\sum_{j=-\infty}^{+\infty}K(t-\tau-jL_1)=K_{L_1}(t-\tau).
\end{equation*}
\end{proof}

\subsection{Extension problem}
The boundary reaction problem given in \eqref{equation1} defines the conformal fractional Laplacian in $\r^n$ for the Euclidean metric $|dx|^2$ as
\begin{equation*}
P_{\gamma}^{|dx|^2}w=-\tilde{d}_{\gamma}\lim_{y\rightarrow 0}y^{a}\partial_{y} W= (-\Delta)^{\gamma}w.
\end{equation*}
In the curved case one may define the conformal fractional Laplacian $P^g_{\gamma}$ with respect to a metric $g$ on a $n-$dimensional manifold $M$. $P^g_{\gamma}$ a pseudo-differential operator of order $2\gamma$ (see \cite{MarChang},\cite{CaseChang}).
It satisfies that, for any conformal metric $$g_{w}:=w^{\frac{4}{n-2\gamma}}g,\ w>0,$$ we have
\begin{equation*}
 P_{\gamma}^{g_{w}}f=w^{-\frac{n+2\gamma}{n-2\gamma}}P_{\gamma}^{g}(wf), \quad \forall f\in \mathcal C^{\infty}(M).
\end{equation*}

This conformal property allows to formulate an equivalent extension problem to \eqref{equation1} for the function $v$ defined as in  \eqref{wv}. In particular, in \cite{paper1} the authors consider the geometric interpretation of problem \eqref{equation3}. They choose $M=\mathbb R\times\mathbb S^{n-1}$ with the metric $g_0=dt^2+g_{\s^{n-1}}$ and they extend the problem with the new positive variable $\rho$, that we will call defining function, to an $(n+1)-$dimensional manifold $X^{n+1}=M\times (0,2)$, with the extended metric \begin{equation*}
\bar{g}=d\rho^2+\left(1+\tfrac{\rho^2}{4}\right)^2dt^2+\left(1-\tfrac{\rho^2}{4}\right)^2 g_{\s^{n-1}},\end{equation*}
for  $\rho\in(0,2)$ and $t\in\mathbb R$. Then \eqref{equation1} is equivalent to
\begin{equation}\label{fracyamextv}\left\{
\begin{split}
-\divergence_{\bar{g}}(\rho^{1-2\gamma}\nabla_{\bar{g}} V)+ E(\rho)V&=0\ \text{ in } (X^{n+1},\bar g),\\
V&=v \text{ on }\{\rho=0\},\\
-\tilde{d}_{\gamma}\lim_{\rho\rightarrow 0}\rho^{1-2\gamma}\partial_{\rho}V&=c_{n,\gamma}v^{\beta}\text{ on }\{\rho=0\},
\end{split}
\right.\end{equation}

Finally it is known (see \cite{MarChang}) that there exists a new defining function $\rho^*=\rho^*(\rho)$ such that the problem \eqref{fracyamextv} can be rewritten on the extension $X^*=M\times(0,\rho_0^*)$, as
\begin{equation}\label{Yamfracspec}\left\{
\begin{split}
-\divergence_{g^*}((\rho^*)^{1-2\gamma}\nabla_{g^*}V)&=0\ \text{ in }(X^*,g^*),\\
V&=v\ \text{ on }\{\rho^*=0\},\\
-\tilde{d}_{\gamma}\lim_{\rho^*\to 0}(\rho^*)^{1-2\gamma}\partial_{\rho^*}V +c_{n,\gamma}v&=c_{n,\gamma}v^{\beta}\ \text{ on }\{\rho^*=0\},
\end{split}\right.
\end{equation}
where $g^*=\frac{(\rho^*)^2}{\rho^2}\bar{g}$. %
We look for  radially symmetric solutions $v=v(t)$, $V=V(t,\rho)$ of \eqref{Yamfracspec}. For such solutions we have that $\mathscr{L}_{\gamma}$ is the Dirichlet-to-Neumann operator for this problem, i.e.,
$$\mathscr{L}_{\gamma}(v)=
%P^{g_0}_{\gamma}=
-\tilde{d}_{\gamma}\lim_{\rho^*\to 0}(\rho^*)^{1-2\gamma}\partial_{\rho^*}V +c_{n,\gamma}v.$$

\section{Technical results}
\subsection{Function Spaces}
\begin{defi}
We shall work with the following function space
\begin{equation*}
\begin{split}
H_L^\gamma = \{&v:\R\rightarrow\R; \ v(t+L)= v(t)\text{ and }\\ &\left.\int_0^L \int_0^L (v(t)- v(\tau))^2 K_L ( t-\tau) d \tau dt +\int_0^L v(t)^2 d t <+\infty \right\}
\end{split}
\end{equation*}
 with the norm given by
 \begin{equation*}
\|v\|_{H^\gamma_L}
=\left(\int_0^Lv(t)^2\,dt+\int_0^L\int_0^L|v(t)-v(\tau)|^2K_L(t-\tau)\,dt\,d\tau\right)^{1/2}.
\end{equation*}
Note that we will denote
\begin{equation*}
\begin{split}
W^{\gamma,p}_L=\{&v:\r\rightarrow\r;\ v(t+L)=v(t) \text{ and }\\ &\|v\|^p_{L^p(0,L)}+\int_0^L\int_0^L\frac{|v(t)-v(\tau)|^p}{|t-\tau|^{1+\gamma p}}\,dt\,d\tau< \infty\},
\end{split}
\end{equation*}
with the norm given by
\begin{equation*}
\|v\|_{W^{\gamma,p}_L}=\left(\|v\|^p_{L^p(0,L)}+\int_0^L\int_0^L\frac{|v(t)-v(\tau)|^p}{|t-\tau|^{1+\gamma p}}\,dt\,d\tau\right)^{1/p},
\end{equation*}
which is equivalent to the norm
$$\|v\|_{\tilde{W}^{\gamma,p}_L}=\left(\|v\|^p_{L^p(0,L)}+\int_0^L\int_0^L |v(t)-v(\tau)|^p K(t-\tau)\,dt\,d\tau\right)^{1/p},$$
for the kernel $K$ given in \eqref{K}.
\end{defi}

Now we are going to introduce some fractional inequalities, continuity and compactness results whose proofs for an extension domain can be found in \cite{hitchhiker}. Here we are working with periodic functions, which avoids the technicalities of extension domains but the same proofs as in \cite{hitchhiker} are valid.
\begin{prop}{(Fractional Sobolev inequalities.)}\label{comp1} (Theorems $6.7$ and $6.10$, \cite{hitchhiker})
Let $\gamma\in(0,1)$, $p\in[1,+\infty)$ such that $\gamma p\leq 1$ and $p^*=\tfrac{np}{n-\gamma p}$.
 Then there exists a positive constant $C=C(p,\gamma)$ such that, for any $v\in W^{\gamma,p}_L$, we have
 \begin{equation*}
 \|v\|_{L^q(0,L)}\leq C\|v\|_{W^{\gamma,p}_L},
 \end{equation*}
 for any $q\in[1,p^*);$ i.e., the space $W^{\gamma,p}_L$ is continuously embedded in $L^q(0,L)$ for any $q\in[1,p^*)$.
\end{prop}

\begin{prop}{(Compact embeddings)}\label{comp2} (Theorem $7.1$ and Corollary $7.2$, \cite{hitchhiker}.)
Let $\gamma\in(0,1)$ and $p\in[1,+\infty)$, $q\in[1,p]$,
 and \emph{J} be a bounded subset of $L^p(0,L)$. Suppose
 \begin{equation*}
 \sup_{f\in\text{\emph{J}}}\int_{[0,L]}\int_{[0,L]}\frac{|f(x)-f(y)|^p}{|x-y|^{n+\gamma p}}\,dx\,dy<+\infty.
 \end{equation*}
 Then \emph{J} is pre-compact in $L^q(0,L)$. %\\
 Moreover, if $\gamma p<1$, then \emph{J} is pre-compact in $L^q(0,L)$, for all $q\in[1,p^*)$.
\end{prop}

\begin{rem}\label{gamma1/2}
If $\gamma=1/2$, we have the compact embedding
$$W^{1/2,2}_L\subset\subset L^q(0,L),\ \text{ for } q\in(1,\infty).$$
Indeed, a consequence of Proposition \ref{comp1} is $W_L^{1/2,2}\subset W_L^{\gamma,2},\ \forall\gamma <1/2$, thus Proposition \ref{comp2} provides
$$W_L^{1/2,2}\subset W_L^{\gamma,2}\subset\subset L^q(0,L),\quad \forall q\in(1,\tfrac{2}{1-2\gamma}),\gamma<1/2.$$
%continuo con compacto es compacto
We conclude by letting $\gamma\rightarrow 1/2$.
\end{rem}
\begin{prop}{(H\"older fractional regularity.) }(Theorem $8.2$ in \cite{hitchhiker}.) \label{holder}
Let $p\in[1,+\infty)$, $\gamma\in(0,1)$ such that $\gamma p >1$. Then there exists $C>0$, depending on $\gamma$ and $p$, such that
\begin{equation*}
\|v\|_{\mathcal{C}^{0,\alpha}([0,L])}\leq C \left(\|v\|^p_{L^p(0,L)}+\int_{0}^L\int_{0}^L\tfrac{|v(t)-v(\tau)|^p}{|t-\tau|^{1+\gamma p}}\,dt\,d\tau\right)^{1/p}
\end{equation*}
for any $L$-periodic function $v\in L^p(0,L)$, with $\alpha=\gamma-1/p$.
\end{prop}
Note that with the equi-continuity given in Proposition \ref{holder} we can apply Ascoli-Arzel{\'a} to show the compactness
\begin{equation*}
  W^{\gamma,2}_L \subset\subset L^q(0,L)\ \forall q\in(1,\infty)\text{ with }\gamma>1/2.
 \end{equation*}

\begin{rem}\label{compactness}
 We have the compact embedding
\begin{equation*}
 H_L^{\gamma} \subset\subset L^q(0,L), \ \forall \gamma\in(0,1),
 \end{equation*}
 where \begin{equation}\label{q}
 q\in(1,\tfrac{2}{1-2\gamma}) \text{ if } \gamma\leq\tfrac{1}{2}\quad
  \text{and} \quad
  q\geq 1\text{ if }\gamma>\tfrac{1}{2}.
 \end{equation}
 Indeed, Proposition \ref{comp2}, Remark \ref{gamma1/2} and Proposition \ref{holder} with  provide $W_L^{\gamma,2} \subset\subset L^q(0,L)$ for all $\gamma\in(0,1)$ and $q$ as in \eqref{q}. %
 But from the definition of $K_L$ given in \eqref{KL} and the positivity of the function $K$, %
 we have the following inequality between norms
 $$
  \|v\|_{W^{\gamma,2}_L}\leq \|v\|_{H^{\gamma}_L}.$$
 \end{rem}
\begin{prop}{(Poincare's fractional inequality.) }%\cite{Hurri-Syrjanen_Vahakangas}
Let $v\in H^{\gamma}_L$ with zero average (i.e. $\int_{0}^L v(t)\,dt=0$), then there exists $c>0$ such that
\begin{equation}\label{poincare}
\|v\|^2_{L^2(0,L)}\leq c\int_{0}^L\int_{0}^L\frac{(v(t)-v(\tau))^2}{|t-\tau|^{1+2\gamma}}\,dt\,d\tau.
\end{equation}
\end{prop}

\begin{proof}
Inspired on the proof of the classical Poincare's inequality given in Theorem $7.16$ in \cite{Salsa}, we prove \eqref{poincare}.
By contradiction assume that, $\forall  j\geq 1$, there exists $v_j\in H^{\gamma}_L$ satisfying
\be\label{1p}
\|v_j\|^2_{L^2(0,L)}> j\int_{0}^L\int_{0}^L\frac{(v_j(t)-v_j(\tau))^2}{|t-\tau|^{1+2\gamma}}\,dt\,d\tau.
\ee
On the one hand, we normalize $v_j$ in $L^2(0,L)$ by
$w_j:=\tfrac{v_j}{\|v_j\|_{L^2(0,L)}}$, so $\|w_j\|_{L^2(0,L)}=1$.
Because of \eqref{1p} it follows that
\be\label{2p}
\int_{0}^L\int_{0}^L\frac{(w_j(t)-w_j(\tau))^2}{|t-\tau|^{1+2\gamma}}\,dt\,d\tau<\tfrac{1}{j}\leq 1,
\ee
that is, $\{w_j\}$ is bounded in the $H^{\gamma}_L$ norm. By the compactness from Remark \ref{compactness}, we  obtain a subsequence $\{w_i\}$ that converges strongly in $L^{2}(0,L)$, i.e, there exists $w\in L^2(0,L)$ such that $w_i\rightarrow w\text{ in }L^2(0,L)$. Thus,
$$\|w\|_{L^2(0,L)}=\lim_{j\to\infty}\|w_j\|_{L^2(0,L)}=1.$$
On the other hand, also by the compactness given in Remark \ref{compactness}, we have weak semiconvergence in $H^{\gamma}_L$. Thus the following inequality follows
\begin{equation*}\int_{0}^L\int_{0}^L\frac{(w(t)-w(\tau))^2}{|t-\tau|^{1+2\gamma}}\,dt\,d\tau\leq \liminf_{j\rightarrow \infty}\int_{0}^L\int_{0}^L\frac{(w_j(t)-w_j(\tau))^2}{|t-\tau|^{1+2\gamma}}\,dt\,d\tau.\end{equation*}
Thanks to \eqref{2p}, this gives $$\int_{0}^L\int_{0}^L\frac{(w(t)-w(\tau))^2}{|t-\tau|^{1+2\gamma}}\,dt\,d\tau=0,$$
that is, $w$ must be constant and, since has zero average, it has to be the zero function.
\end{proof}

\subsection{Maximum principles}

\begin{prop}{(Strong maximum principle).}\label{strongmaxppo}
Let $v\in H_L^{\gamma,2}\cap\mathcal{C}^{0}(\r)$ with $v\geq 0$ be a solution of $$\mathcal{L}_{\gamma}v= f(v),\quad \mbox{in} \ \R,$$
where $f$ satisfies $f(v)\geq 0$ if $v\geq 0$. Then $v>0  $ or $v\equiv 0$.
\end{prop}
\begin{proof}
Since $v\geq 0$, we have that \begin{equation}\label{poslap}
\mathscr{L}_{\gamma}v=f(v)\geq 0.
\ee
Suppose that there exists a point $t_0\in\R $ with $v(t_0)=0$, then
\begin{equation*}\begin{split}
\mathscr{L}_{\gamma}v (t_0)&=\kappa_{n,\gamma}\text{P.V}\int_{-\infty}^{+\infty}(v(t_0)-v(\tau))K(t_0-\tau)\,d\tau+c_{n,\gamma}v(t_0) \\ &=\kappa_{n,\gamma}\text{P.V}\int_{-\infty}^{+\infty}(-v(\tau))K(t_0-\tau)\,d\tau\leq 0
\end{split}
\end{equation*}
satisfies \eqref{poslap} only in the case $v\equiv 0$.
\end{proof}

\subsection{Regularity}

In the following Proposition \ref{reg1} we concentrate on the local regularity, using the equivalent characterization for $\mathscr{L}_{\gamma}$ as a Dirichlet-to-Neumann operator for problem \eqref{Yamfracspec}. First, we fix some notation that we will use here. Let $0<R<\rho_0^*$, we denote
\begin{equation*}
\begin{split}B_R^+&=\{(t,\rho^*)\in\r^2\,:\, \rho^*>0, |(t,\rho^*)|<R\},\\
\Gamma_R^0&=\{(t,0)\in\partial\r^2_+ \,:\,  |t|<R\}.
\end{split}
\end{equation*}

\begin{prop}\label{reg1}
Fix $\gamma<1/2$ and let $V=V(t,\rho^*)$ be a solution of the extension problem
\begin{equation}\label{Yamfracspeclocal}\left\{
\begin{split}
-\divergence_{g^*}((\rho^*)^{1-2\gamma}\nabla_{g^*}V)&=0\ \text{ in }(B_{2R}^+,g^*),\\
-\tilde{d}_{\gamma}\lim_{\rho^*\to 0}(\rho^*)^{1-2\gamma}\partial_{\rho^*}V +c_{n,\gamma}v&=c_{n,\gamma}v^{\beta}\ \text{ on }\Gamma_{2R}^0.
\end{split}\right.
\end{equation}
If $$\int_{\Gamma_{2R}^0}|v|^{\tfrac{2}{1-2\gamma}}\,dt=:\zeta <\infty,$$ then for each $ {p}>1$, there exists a constant $C_{ {p}}=C( {p},\zeta)>0$ such that
 \begin{equation*}
 \sup_{B^+_R}|V|+\sup_{\Gamma^0_R}|v|\leq C_{ {p}}\left[\left(\tfrac{1}{R^{n+1+a}}\right)^{1/ {p}}\|V\|_{L^{ {p}}(B^+_{2R})}+
 \left(\tfrac{1}{R^{n}}\right)^{1/ {p}}\|v\|_{L^{ {p}}(\Gamma^0_{2R})}\right].
 \end{equation*}
\end{prop}

\begin{proof}
This $L^\infty$ bound is proven for linear right hand side in Theorem $2.3.1$ in \cite{FabesKenigSerapioni}. A generalization for the nonlinear subcritical case is given in Theorem $3.4$ in \cite{MarQing}.
 Here we can follow the same proof as in \cite{MarQing} because we have reduced our problem to one-dimensional problem for $t\in\partial\mathbb R^2_+$ and thus, $\beta=\frac{n+2\gamma}{n-2\gamma}$ is a subcritical exponent.
\end{proof}

The following two propositions could be also proved using the extension problem \eqref{Yamfracspeclocal}. However, they can be phrased in terms of a general convolution kernel, as we explain here. Thus we fix $K:\r\rightarrow [0,\infty)$ a measurable kernel satisfying:
\begin{enumerate}
 \item[a)] $\nu\leq K(t)|t|^{1+\tfrac{\gamma}{2}}\leq \nu^{-1}$ a.e $t\in\r$ with $|t|\leq 1$,
 \item[b)] $K(t)\leq M |t|^{-n-\eta}$ a.e. $t\in\r$ with $|t|>1$,
                                                          \end{enumerate}
for some $\gamma\in(0,1)$, $\nu\in(0,1)$, $\eta>0$, $M\geq 1$. Consider the functional defined in \eqref{L} by $$ (\mathscr{L}_{\gamma}v)(t)=\kappa_{n,\gamma}\text{P.V}\int_{-\infty}^{+\infty}(v(t)-v(\tau))K(t-\tau)\,d\tau+c_{n,\gamma}v,$$
for $v\in L^{p}(\r)$. We study the regularity of solutions to
\be\label{eqlinear}
\mathscr{L}_{\gamma}v=f.
\ee

\begin{prop}\label{reg2}
Let $f\in L^q$ for some $q>n$ and $v$ solution of \eqref{eqlinear} in $B_R(x_0)$, then there exist constants $c>0$ and $\alpha\in(0,1)$ which depend on $n$, $\nu$, $M$, $\eta$, $\gamma$, $q$ and $A$, and remain positive as $\gamma\rightarrow 1$,  such that for any $R\in(0,1)$,
\begin{equation*}
|v(t)-v(\tau)|\leq c|t-\tau|^{\alpha}\left(R^{-\alpha}\|v\|_{L^\infty}+\|f\|_{L^q}\right).
\end{equation*}
\end{prop}
\begin{proof}
Since our kernel corresponds to a tempered stable process, this regularity was given by Kassmann in his article \cite{Kassmann352}  (see Theorem $1.1$ and Extension $5$). We could also follow the same steps as for Theorem $5.1$ in \cite{silvestre} since Lemma $4.1$ and Remark $4.3$ in this paper \cite{silvestre} hold for our $K$ (note the expansion in Lemma \ref{Kexpansion}).
\end{proof}

\begin{prop}\label{reg3}
Let $\alpha\in(0,1)$. Assume $f\in \mathcal{C}^{\alpha}(\r)$, and let $v\in L^{\infty}(\r)$ be a solution of \eqref{eqlinear} in $\r^n$. Then there exists $c>0$ depending on $n,\alpha,\gamma$ such that
\begin{equation*}
\|v\|_{\mathcal{C}^{\alpha+2\gamma}}\leq c\left(\|v\|_{\mathcal{C}^{\alpha}}+\|f\|_{\mathcal{C}^{\alpha}}\right).
\end{equation*}

\end{prop}

\begin{proof}
Under our assumptions, on the one hand, Dong and Kim proved in Theorem $1.2$ from \cite{DongKim_SchauderNonlocal} that $(-\Delta)^{\gamma} v\in \mathcal{C}^{\alpha}$ and moreover the following estimate holds:
 \be\label{des1}
\|(-\Delta)^{\gamma} v\|_{\mathcal{C}^{\alpha}}\leq c\left(\|v\|_{\mathcal{C}^{\alpha}}+\|f\|_{\mathcal{C}^{\alpha}}\right).
\ee
On the other hand, Silvestre in Proposition $2.8$ in \cite{Silvestre2017CPAM}, showed that
\begin{itemize}
\item If $\alpha+2\gamma\leq 1$, then $v\in \mathcal{C}^{\alpha+2\gamma}$ and
\be \label{des2}
\|v\|_{\mathcal{C}^{\alpha+2\gamma}(\r)}\leq c(\|v\|_{L^{\infty}}+\|(-\Delta)^{\gamma} v\|_{\mathcal{C}^{\alpha}}).
\ee
\item If $\alpha+2\gamma > 1$, then $v\in \mathcal{C}^{1,\alpha+2\gamma-1}$ and
\be \label{des2.2}
\|v\|_{\mathcal{C}^{1,\alpha+2\gamma-1}(\r)}\leq c(\|v\|_{L^{\infty}}+\|(-\Delta)^{\gamma} v\|_{\mathcal{C}^{\alpha}}).
\ee
\end{itemize}
Thus, combining \eqref{des1} with \eqref{des2} and \eqref{des2.2} we have the claimed regularity.
\end{proof}

\begin{rem}\label{regularity}
The previous Propositions \ref{reg1}, \ref{reg2}, \ref{reg3} imply that for $\gamma<1/2$ any $v\in L^{\beta+1}$ solution of equation \eqref{10} %
 satisfies $v\in \mathcal{C}^{\infty}$. A standard argument yields the same conclusion for $\gamma=1/2$ too. Finally,
 if $\gamma> 1/2$ Proposition \ref{holder} automatically implies that any function $v\in H^{\gamma}_L$ also satisfies $v\in \mathcal{C}^{\infty}$.
\end{rem}

\subsection{Subcritical case.}
Note that the following Lemma \ref{scri} has been studied by different authors if $N>2\gamma$, even for $1<p<\tfrac{N+2\gamma}{N-2\gamma}$ (see \cite{ChenLiZhang,ChenLiOu,YYLi}), but in this paper we need this result also for $2\gamma\geq N$ since  we have reduced our problem to dimension $N=1$ for any $\gamma\in(0,1)$. We will use it for $p=\frac{n+2\gamma}{n-2\gamma}$.

\begin{lem}\label{scri}
Let $w$ be solution for
\begin{equation}\label{subc}
(-\Delta)^{\gamma}w=w^p,\quad 0\leq w \leq 1,\quad p>1, \ \ (N-2\gamma)p<N.
\end{equation}
Then $w\equiv 0$.
\end{lem}
\begin{proof}
Let $\eta$ be a smooth function. In fact we may choose
\begin{equation}\label{eta}\eta=(1+|x|)^{-m},\quad \text{where }m=N+2\gamma.\end{equation}
Then multiplying \eqref{subc} by the test function $\eta$, integrating over $\r^N$ and using integration by parts in the right hand side of \eqref{subc} we obtain the following inequality
         \begin{equation}\label{equation100}\begin{split}
         \left|\int_{\r^N}w^p\eta\,dx\right|&= \left|\int_{\r^N}\left(w(x)\int_{\r^N}\frac{\eta(x)-\eta(y)}{|x-y|^{N+2\gamma}}\,dy\right)\,dx\right|\\
         &\leq  \left|\int_{\r^N} \left((w(x)\eta^{1/p}(x))\eta(x)^{-1/p}\int_{\r^N}\frac{\eta(x)-\eta(y)}{|x-y|^{N+2\gamma}}\,dy\right)\,dx\right|\\
        & \leq  \left|\int_{\r^N}w^p(x)\eta(x)\,dx\right|^{1/p}\left(\int_{\r^N}\left|(\eta(x)^{-1/p}(-\Delta)^{\gamma}\eta(x))^{p/(p-1)}\right|\,dx\right)^{(p-1)/p}.
         \end{split}\end{equation}
We just need to compute the second term in the right hand side. Firstly we can check that it is bounded. Since
\begin{equation}\label{2term}
\eta(x)^{-\frac{1}{p-1}}|(-\Delta)^{\gamma}\eta(x)|^{\frac{p}{p-1}}\leq c(1+|x|)^{(N+2\gamma)\frac{1}{p-1}}(1+|x|)^{-\frac{p}{p-1}(N+2\gamma)}\leq(1+|x|)^{-(N+2\gamma)},
\end{equation}
 we have
\begin{equation*}
\int_{\r^N}\eta(x)^{-\frac{1}{p-1}}|(-\Delta)^{\gamma}\eta(x)|^{\frac{p}{p-1}}\,dx<\infty.
\end{equation*}
Note that for inequality \eqref{2term} we have used the definition of the test function given in \eqref{eta} and the following bound
\begin{equation}\label{bou}
|(-\Delta)^{\gamma}\eta|\leq c(1+|x|)^{-(N+2\gamma)},\quad \text{for }x\text{ large enough};
\end{equation}
which is proven at the end of the proof of this Lemma.
Now we chose $$\eta_R(x)=\eta(x/R).$$
Performing a similar analysis to that of \eqref{equation100}, we obtain
\begin{equation*}
\int_{\r^N}w^p(x)\eta_R(x)\leq \int_{\r^N}\eta_R(x)^{-1/(p-1)}\left|\int_{\r^N}\frac{\eta(x)-\eta(y)}{|x-y|^{N+2\gamma}}\,dy\right|^{p/(p-1)}\,dx.
\end{equation*}
Then, by scaling,
\begin{equation*}
\int_{|x|\leq R}w^p(x)\leq cR^{N-\frac{2p\gamma}{p-1}}\int_{\r^N}\eta(x)^{-1/(p-1)}\left|\int_{\r^N}\frac{\eta(x)-\eta(y)}{|x-y|^{N+2\gamma}}\,dy\right|^{p/(p-1)}\,dx.
\end{equation*}
Note that $N-\tfrac{2p\gamma}{p-1}<0$ by hypothesis. Then, letting $R$ tend to infinity, we obtain
$$\int_{|x|\leq R}w^p(x)\,dx\rightarrow 0\text{ as }R\rightarrow+\infty.$$
Therefore, we have $w\equiv 0$.\\

In order to conclude we just need to check inequality \eqref{bou} before. It follows from standard potential analysis. In fact, for $|x|\geq 1$ we have that
$$\left|(-\Delta)^{\gamma}\eta(x)\right|=\left|\text{P.V.}\int_{\r^N}\frac{\eta(x)-\eta(y)}{|x-y|^{N+2\gamma}}\,dy\right|\leq |I_1|+|I_2|+|I_3|+|I_4|,$$
where these integrals can be bounded as follows:  for the first integral we use that $|x-y|$ is small enough to check that
        \begin{equation*}\begin{split}
        |I_1|&=\left|\text{P.V.}\int_{|x-y|<1}\frac{\eta(x)-\eta(y)}{|x-y|^{N+2\gamma}}\,dy\right|
        =\left| \int_{|x-y|<1}\frac{\eta(x)-\eta(y)-\eta'(x)|x-y|}{|x-y|^{N+2\gamma}}\,dy\right|\\
        &\leq C \int_{|x-y|<1}\frac{|\eta''(x)||x-y|^2}{|x-y|^{N+2\gamma}}\,dy
%\leq  \frac{C}{(1+|x|)^{N+2}}
\leq  \frac{C}{(1+|x|)^{N+2\gamma}}.
        \end{split}\end{equation*}
        For the second one, we have that $|x-y|<\frac{|x|}{2}$, then, we can use that $$|\eta(x)-\eta(y)|\leq|\eta'(\xi)||x-y|\leq C(1+|x|)^{-(N/2+2\gamma-1)}|x-y|,$$ and bound the integral as follows
        \begin{equation*}\begin{split}
         |I_2|&=\left|\int_{1<|x-y|<\frac{|x|}{2}}\frac{\eta(x)-\eta(y)}{|x-y|^{N+2\gamma}}\,dy\right|\\
        %&\leq  Esta linea esta mal, aunque la conclusion es cierta c\left(\int_{1<|x-y|<\frac{|x|}{2}}\frac{1}{|x-y|^{N+2\gamma}}\,dy\right)(1+|x|)^{-(N/2+2\gamma-1)}|x-y|\\
        &\leq C|x|^{1-2\gamma}(1+|x|)^{-(N/2+2\gamma-1)}\leq \frac{C}{(1+|x|)^{N+2\gamma}},
        \end{split}\end{equation*}
        since $x$ is large enough and $|x|\sim|y|$, indeed $|y|\geq |x|-|x-y|\geq\frac{|x|}{2}$ and $|y|\leq \frac{3}{2}|x|$.

The third one is directly bounded, \begin{equation*}\begin{split}
            |I_3|=&\left|\int_{\frac{|x|}{2}<|x-y|<2|x|}\frac{\eta(x)-\eta(y)}{|x-y|^{N+2\gamma}}\,dy\right|
            \leq\frac{2^{N+2\gamma}}{|x|^{N+2\gamma}}\left|\int_{\frac{|x|}{2}<|x-y|<2|x|}(\eta(x)-\eta(y))\,dy\right|\\
           \leq &\frac{2^{N+2\gamma}}{|x|^{N+2\gamma}}\left|\eta(x)|x|^{-N}-\int_{\frac{|x|}{2}<|x-y|<2|x|}\eta(y)\,dy\right|
           \leq\frac{C}{|x|^{N+2\gamma}}\sim\frac{C}{(1+|x|)^{N+2\gamma}},
           \end{split}
           \end{equation*}
using that $|x|$ is large enough.

    For the fourth and last one, we use that $|y|\geq|x-y|-|x|\geq|x|$, then
     \begin{equation*}\begin{split}
     |I_4|=&\left|\int_{|x-y|>2|x|}\left(\frac{\eta(x)-\eta(y)}{|x-y|^{N+2\gamma}}\right)\,dy\right|\leq C\left(\int_{|x-y|>2|x|}\frac{1}{|x-y|^{N+2\gamma}}\,dy\right)(1+|x|)^{-(N+2\gamma)}\\
\leq& \frac{C}{(1+|x|)^{(N+2\gamma)}}.
     \end{split}\end{equation*}
\end{proof}

\section{Proof of Theorem \ref{th1}}
\subsection{Variational  Formulation}
We consider the following minimization problem
\be\label{cL}
c(L)= \inf_{ v \in H_L^\gamma, v \not \equiv 0}\mathscr{F}_L(v),
\ee
 where
\be\label{functional}
\mathscr{F}_L(v)=\frac{ \kappa_{n,\gamma}\int_0^L \int_0^L (v(t)-v(\tau))^2 K_L (t-\tau)\, dt \,d\tau +c_{n,\gamma} \int_0^L v(t)^2 \,dt}{ (\int_0^L v(t)^{\beta+1} dt)^{\frac{2}{\beta+1}}}.
\ee

Our first lemma shows that
\begin{lem}
\label{Lem1}
For any $L>0$, $ c(L)$ is achieved by a positive function $v_L\in \mathcal{C}^{\infty}$ which solves \begin{equation}\label{13_2}
\mathscr{L}^L_{\gamma}v=\kappa_{n,\gamma}P.V. \int_0^L (v(t)- v(\tau)) K_L (t-\tau) d \tau + c_{n,\gamma} v = c_{n,\gamma}v^{\beta}, \ \ \  \mbox{where}\ \ \beta=\tfrac{n+2\gamma}{n-2\gamma}.
\end{equation}
\end{lem}

\begin{proof}
Considering that the value of multiplicative constants does not affect this proof, we may assume that $c_{n,\gamma}=1$ and $\kappa_{n,\gamma}=1$. %\\
Since $c(L)$ is invariant by rescaling we can assume that
\begin{equation}\label{a1}
\int_0^Lv^{{\beta}+1}\,dt= 1;
\end{equation}
 thus $\mathscr{F}_L[v]=\|v\|^2_{H^{\gamma}_L}$. %\\
First note that if $c(L)$ is achieved by a function $v_L$, then this function solves \eqref{13_2} because this is the Euler-Lagrange equation for the functional \eqref{functional}.

By construction, the functional $\mathscr{F}_L(v)$ is non-negative and therefore it is bounded from below, so the infimum is finite. Next we show that a minimizer exists.
 Let $\{v_i\}$ be a minimizing sequence normalized to satisfy \eqref{a1}, %
 such that $\mathscr{F}_L(v_i)\leq c(L)+1$. %where $c(L)$ is defined in \eqref{cL}.
Because of Remark \ref{compactness}, for all $\gamma\in(0,1)$ we have the compact embedding of $H^{\gamma}_L$ in $L^{q}$, with $q\in(1,\tfrac{2}{1-2\gamma})$ if $\gamma\leq\tfrac{1}{2}$ and $q\geq 1$ if $\gamma>\tfrac{1}{2}$ so, in particular, for $q=\beta+1$.  Moreover, there exists $v_L\in H_L^{\gamma}$ such that $v_i\rightharpoonup v_L$. This implies
\begin{equation}\label{desnormas}
\|v_L\|_{H^{\gamma}_L}\leq\liminf_{j}\|v_j\|_{H_L^{\gamma}}.
\end{equation}
Since $\{v_i\}$ is a minimizing sequence, $\liminf \|v_j\|_{H_L^{\gamma}}=c(L)$, and \eqref{desnormas} implies that
we have a minimizer $v_L\in H_L^{\gamma}$.
 The compact embedding assures that convergence is strong in $L^{\beta+1}$, i.e., $$1=\lim_j\|v_j\|_{L^{{\beta}+1}}=\|v_L\|_{L^{{\beta}+1}}.$$
Now we apply Remark \ref{regularity} to obtain $v_L\in \mathcal{C}^{\infty}$.

Finally we observe that the minimizer $v_L\in H^{\gamma}_L$ must be positive. If $v_L$ is not non-negative we take $w=|v_L|\in H_L^{\gamma}$ and the following inequality holds
\be\label{posit}
\mathscr{F}_L(w)\leq \mathscr{F}_L(v_L),
\ee
obtaining a contradiction.
Indeed if sign$(v(t))=$ sign$(v(\tau))$, equality holds in \eqref{posit} and if sign$(v(t))\neq $ sign$(v(\tau))$, \eqref{posit} is also true because
\begin{equation*}
\begin{split}
(w(t)-w(\tau))^2&=(v_L(t)+v_L(\tau))^2\leq\max\{(v_L(t))^2,(v_L(\tau))^2\}\\
&\leq(|v_L(t)|+|v_L(\tau)|)^2 =(v_L(t)-v_L(\tau))^2.
\end{split}
\end{equation*}
Once we have the non-negativity of the minimizer, since $\|v_L\|_{L^{\beta}}=1$, the maximum principle given in Proposition \ref{strongmaxppo} applied to equation \eqref{13_2} assures that $v_L> 0$. %\\
Therefore we conclude the proof of the Lemma \ref{Lem1}.
%\
%
\end{proof}
    \medskip

    We now introduce the weak formulation of the problem. We will say that $v\in H^{\gamma}_L$ is weak solution of \eqref{13_2} if it satisfies
    \begin{equation}\label{formdebil}
    \langle\mathscr{L}_{\gamma}^Lv,\phi\rangle= c_{n,\gamma}\int_0^Lv^{\beta}(t)\phi(t)\,dt,\quad  \forall \phi\in H^{\gamma}_L
    \end{equation}
where $\langle \, , \, \rangle$ is defined by
    \begin{equation*}
    \langle\mathscr{L}_{\gamma}^Lv,\phi\rangle=\kappa_{n,\gamma}P.V.   \int_{0}^L\int_{0}^L (v(t)-v(\tau))(\phi(t)-\phi
    (\tau))K_L(t-\tau)\,dt\,d\tau + c_{n,\gamma} \int_0^L v(t)\phi(t)\,dt.
    \end{equation*}

\subsection{Proof of Theorem \ref{th1}:}
At this moment it is unclear if the minimizer $v_L$ for \eqref{functional} is the constant solution. %
Let \begin{equation*} c^{*} (L)= c_{n,\gamma}L^{ \frac{\beta-1}{\beta+1}}
\end{equation*} be the energy of the constant solution.
The next key lemma provides a criteria:

\begin{lem}
\label{L21}
Assume that $ c(L_1)$ is attained by a nonconstant function $v_{L_1}$.  Then $ c(L) < c^{*} (L)$ for all $L> L_1$.
\end{lem}

%\noindent
%{\bf Proof:}
\begin{proof}
Let $v_{L_1}$ be the minimizer for $L_1$, then $v_{L_1}$ is the solution to
\begin{equation*}
\mathscr{L}_{\gamma}^{L_1}(v_{L_1}):=\kappa_{n,\gamma}\int_0^{L_1} (v_{L_1} (t)- v_{L_1} ({\tau})) K_{L_1} ( t-{\tau}) d {\tau} +c_{n,\gamma} v_{L_1} = c_{n,\gamma}v_{L_1}^{\beta}.
\end{equation*}
By assumption $ v_{L_1} \not \equiv 1$.
Now let $$t=\frac{L_1}{L}\bar{t}\quad\text{ and }\quad
 v(\bar{t})= v_{L_1} \left(\frac{L_1}{L} \bar{t}\right), $$
which is an $L-$periodic function.
By definition it is clear that%
\begin{equation*}\begin{split}c(L) &\leq  \frac{ \kappa_{n,\gamma}\int_0^L \int_0^L (v(\bar{t})-v(\bar{\tau}))^2 K_L (\bar{t}-\bar{\tau}) \,d\bar{t} \,d\bar{\tau} + c_{n,\gamma}\int_0^L v^2(\bar{t}) \,d\bar{t}}{ (\int_0^L v^{{\beta}+1}(\bar{t}) \,d\bar{t})^{\frac{2}{{\beta}+1}}}\\
&= \left(\tfrac{L}{L_1}\right)^{1-\frac{2}{{\beta}+1}} \frac{ \kappa_{n,\gamma}\int_0^{L_1} \int_0^{L_1} ( v_{L_1}(t)- v_{L_1} ({\tau}) )^2 \frac{L}{L_1}  K_L ( \frac{L}{L_1} (t-{\tau})) \,d t \,d {\tau} + c_{n,\gamma}\int_0^{L_1} v_{L_1}^2(t) \,dt}{ (\int_{0}^{L_1} v_{L_1}^{{\beta}+1}(t)\,dt )^{\frac{2}{{\beta}+1}}}
\\&< \left(\tfrac{L}{L_1}\right)^{1-\frac{2}{{\beta}+1}} \frac{ \kappa_{n,\gamma}\int_0^{L_1} \int_0^{L_1} ( v_{L_1}(t)- v_{L_1} ({\tau}) )^2 (  K_{L_1}   (t-{\tau})) \,d t \,d {\tau} + c_{n,\gamma}\int_0^{L_1} v_{L_1}^2(t) \,dt}{ (\int_{0}^{L_1} v_{L_1}^{{\beta}+1}(t)\,dt )^{\frac{2}{{\beta}+1}}}\\
&\leq \left(\tfrac{L}{L_1}\right)^{1-\frac{2}{{\beta}+1}}  c(L_1) \leq \left(\tfrac{L}{L_1}\right)^{1-\frac{2}{{\beta}+1}}  c^{*} (L_1) =c^{*} (L).
\end{split}
\end{equation*}
The second inequality above follows from %the following fact, which has been proved in
 Lemma \ref{appendix}.

Thus we conclude that $c(L) <c^{*}(L)$ for all $L>L_1$ and hence $c(L)$ is attained by a nonconstant minimizer.
\end{proof}

\begin{lem}\label{lemma1} If the period $L$ is small enough, then $c(L)$ is attained by the constant only.
\end{lem}

\begin{proof}
First, we claim that, for $L\leq 1$, the minimizer $v_L$ is uniformly bounded. This follows from  a standard Gidas-Spruck type blow-up argument. In fact, suppose not, we may assume that there exist sequences $\{L_i\}$, $\{v_{L_i}\}$ and $\{t_i\}$ with $t_i\in[0,L_i]$ such that
 $$ \max_{ 0\leq t \leq L_i} v_{L_i} (t)= \max_{t\in \R} v_{L_i} (t)=v_{L_i}(t_i) = M_i \to +\infty.$$
  Note that $v_{L_i}$ satisfies \eqref{13_2}. Now rescale
  \begin{equation*}
  \tilde{t}=\epsilon_i^{-1}(t-t_i),\ \ \tilde{v}_{L_i} (\tilde{t})=\epsilon_i^{\frac{2\gamma}{{\beta}-1}} v_{L_i} ( \epsilon_i \tilde{t}),
\end{equation*}
where $$\ M_i= \epsilon_i^{\frac{-2\gamma}{{\beta}-1}}.$$
With this change of variable, \eqref{13_2} reads
\begin{equation*}
\kappa_{n,\gamma}\int_{\R}\epsilon_i(\tilde{v}_{L_i}(\tilde{t})-\tilde{v}_{L_i}(\tilde{\tau}))K(\epsilon_i(\tilde{t}-\tilde{\tau}))\, d\tilde{\tau} + c_{n,\gamma}\tilde{v}_{L_i}( \tilde{t})=\epsilon_i ^{-2\gamma}c_{n,\gamma}v_{L_i}^{\beta}(\tilde{t}).
\end{equation*}
Because of \eqref{Kexp}
$$ \int_{\R}\epsilon_i(\tilde{v}_{L_i}(\tilde{t})-\tilde{v}_{L_i}(\tilde{\tau}))K(\epsilon_i(\tilde{t}-\tilde{\tau}))\, d\tilde{\tau}\sim \frac{1}{\epsilon_i^{2\gamma}} \int_{\R} \frac{\tilde{v}_{L_i} (\tilde{t})- \tilde{v}_{L_i} (\tilde{{\tau}})}{|\tilde{t}-\tilde{{\tau}}|^{1+2\gamma}} d \tilde{{\tau}} \sim \frac{1}{\epsilon_i^{2\gamma}\kappa_{n,\gamma}} (-\Delta)^{\gamma} \tilde{v}_{L_i}.
$$
Therefore $\tilde{v}_{L_i}$ satisfies
\begin{equation*}
(-\Delta)^{\gamma} \tilde{v}_{L_i}+c_{n,\gamma}\epsilon^{2\gamma}\tilde{v}_{L_i}(\tilde{t})=c_{n,\gamma}\tilde{v}_{L_i}^{\beta}(\tilde{t})+o(1)\text{ as }i\rightarrow\infty.
\end{equation*}
Remark \ref{regularity} assures that all the derivatives of $v_{L_i}$ are equi-continuous functions, thus we can apply Ascoli-Arzel{\'a} theorem to find $ v_{\infty}\in \mathcal{C}^{\infty}$ such that $ \tilde{v}_{L_i} \to v_\infty$ as $ i \to +\infty$ and which satisfies
\begin{equation*}
(-\Delta)^\gamma v_\infty = c_{n,\gamma} v_\infty^{\beta} \ \mbox{in} \ \R.
\end{equation*}
Note that $v_\infty$ is positive. By the result given in Lemma \ref{scri}
 we derive that $ v_\infty\equiv 0$, which contradicts with the assumption that $ v_\infty (0)=1$.
%\noindent

Secondly, we  use Poincare's inequality given in \eqref{poincare} to show that $ v_L \equiv Constant$. In fact we observe that $ \phi =\frac{\partial v_L}{\partial t}$ satisfies
\be
\label{28}
 \mathscr{L}^L_{\gamma}\phi- c_{n,\gamma}{\beta} v_L^{{\beta}-1} \phi=0,
\ee
where $\mathscr{L}^L_{\gamma}$ is defined as in \eqref{LgammaL}.
The weak formulation for the problem from \eqref{formdebil},
the fact that $ v_L$ is bounded  and equation \eqref{28} give
$$ \int_0^L \int_0^L (\phi (t)-  \phi ({\tau}))^2 K_L (t-{\tau}) dt d{\tau} \leq C \int_0^L  \phi^2.
$$
Rescaling $ t=L \tilde{t}, \tilde{\phi} = \phi (L \tilde{t})$ and using \eqref{Kexp}, since $L$ is small enough, we obtain that
$$ \int_0^1 \int_0^1 \frac{( \tilde{\phi} (\tilde{t})-\tilde{\phi} (\tilde{{\tau}}))^2 }{|\tilde{t}-\tilde{{\tau}}|^{1+2\gamma}} d \tilde{t} d \tilde{{\tau}} \leq C L^{2\gamma} \int_0^1 \tilde{\phi}^2.
$$
By Poincare's inequality \eqref{poincare} %
(since $\phi$ has average zero) %
 there exists $C_0>0$ for which
$$
C_0 \int_0^1 \tilde{\phi}^2  \leq \int_0^1 \int_0^1 \frac{( \tilde{\phi} (\tilde{t})-\tilde{\phi} (\tilde{{\tau}}))^2 }{|\tilde{t}-\tilde{{\tau}}|^{1+2\gamma}} d \tilde{t} d \tilde{{\tau}} \leq C L^{2\gamma} \int_0^1 \tilde{\phi}^2,
$$
which yields that
$$
\int_0^1 \tilde{\phi}^2 =0
$$
for $L$ small.
\end{proof}

\begin{lem}\label{lemma2}
If the period $L$ is large enough, then
\be \label{30}
 c(L) < c^{*} (L),
 \ee
 and therefore, we have a non constant positive solution for \eqref{equation3}.
\end{lem}
\begin{proof}
Let
\be\label{b}
b(t):=\left(\frac{e^t}{e^{2t}+1}\right)^{\tfrac{n-2\gamma}{2}},
\ee
which is a ground state solution for \eqref{equation3}. %
This follows because the ``bubble'' \be\label{bubble1}
\omega(x)=\left(\frac{1}{|x|^2+1}\right)^{\tfrac{n-2\gamma}{2}}, \ee  %
 is a solution of \eqref{equation0} that is regular at the origin. Note that  $b(t)>0$ and $b(\pm \infty)=0.$

 Now we take a cut-off function $\eta_L$ which is identically $1$ in the ball of radius $L/4$ and null outside the ball of radius $L/2$. We define a new function
 \begin{equation*}
 v_L(t)=b(t)\eta_L(t).
 \end{equation*}
We will denote $\tilde{v}_L(t)\in H^{\gamma}_L$ the $L-$periodic extension of $v_L$.
The definitions of $c(L)$ and $K_L$, given in \eqref{cL} and \eqref{KL} respectively,
 give us the following equality:
\be\begin{split}\label{cL2}
c(L)%=&\inf_{ v \in H_L^\gamma, v \not \equiv 0}\frac{ \displaystyle{\kappa_{n,\gamma}\sum_{j\in\z}}\int_0^L \int_{-\tau+jL}^{L-\tau+jL} (v(t)-v(\tau))^2 K (s) ds d\tau + c_{n,\gamma}\int_0^L v(t)^2 dt}{ (\int_0^L v(t)^{{\beta}+1} dt)^{\frac{2}{{\beta}+1}}}\\
=&\inf_{ v \in H_L^\gamma, v \not \equiv 0}\frac{\kappa_{n,\gamma}\int_0^L \int_{\r} (v(s+\tau)-v(\tau))^2 K (s) \,ds d\tau + c_{n,\gamma}\int_0^L v(t)^2 dt}{ (\int_0^L v(t)^{{\beta}+1}\, dt)^{\frac{2}{{\beta}+1}}}\\
=&\inf_{ v \in H_L^\gamma, v \not \equiv 0}\frac{\kappa_{n,\gamma}\int_{-L/2}^{L/2} \int_{\r} (v(s+\tau)-v(\tau))^2 K (s) \,ds d\tau + c_{n,\gamma}\int_{-L/2}^{L/2} v(t)^2 \,dt}{ (\int_{-L/2}^{L/2} v(t)^{{\beta}+1} \,dt)^{\frac{2}{{\beta}+1}}},
\end{split}
\ee
where $s:=t-\tau$ and we have used the $L-$periodicity of any $v\in H^{\gamma}_L$. %\\
We use $\tilde{v}_L$ as a test function in the functional \eqref{cL2}. Taking the limit $L\rightarrow\infty$,
\begin{equation*}
\begin{split}
\lim_{L\to\infty}c(L)&\leq\lim_{L\to\infty} \frac{\kappa_{n,\gamma}\int_{-L/2}^{L/2} \int_{\r} (\tilde{v}_L(s+\tau)-\tilde{v}_L(\tau))^2 K (s) \,ds d\tau + c_{n,\gamma}\int_{-L/2}^{L/2} \tilde{v}_L(t)^2 \,dt}{ (\int_{-L/2}^{L/2} \tilde{v}_L(t)^{{\beta}+1} \,dt)^{\frac{2}{{\beta}+1}}}\\
&=\frac{\kappa_{n,\gamma}\int_{\r} \int_{\r} (b(t)-b(\tau))^2 K (t-\tau)\, dt d\tau + c_{n,\gamma}\int_{\r} b(t)^2 \,dt}{ (\int_{\r} b(t)^{{\beta}+1}\, dt)^{\frac{2}{{\beta}+1}}}<\infty,
\end{split}
\end{equation*}
since the ``bubble'' \eqref{bubble1} has finite energy.
 Let us check that all the integrals above are uniformly bounded in order to use the Dominated Convergence Theorem. First, both integrals $\int_{-L/2}^{L/2}\tilde{v}^2_L(t)\,dt$ and $\int_{-L/2}^{L/2}\tilde{v}^{\beta+1}_L(t)\,dt$ are uniformly bounded  since $b(t)\sim e^{-\frac{n-2\gamma}{2}|t|}$. %Next, the behaviour of the kernel given in \eqref{Kinfty} assures that
  Finally, recalling that $b(t),\eta_L\in L^{\infty}$ and the behaviour of the kernel \eqref{Kinfty}
  $$\int_{-L/2}^{L/2} \int_{\r} (\tilde{v}_L(s+\tau)-\tilde{v}_L(\tau))^2 K (s) \,ds d\tau=I_1+I_2,$$
  where %\begin{itemize}
         % \item
          \begin{equation*}
          \begin{split}I_1\sim& \int_{-L/2}^{L/2} \int_{\r \setminus [-\epsilon,\epsilon]} (\tilde{v}_L(s+\tau)-\tilde{v}_L(\tau))^2 e^{-|s|\tfrac{n+2\gamma}{2}} \,ds d\tau\\
           \sim &\int_{\r \setminus [-\epsilon,\epsilon]}e^{-|s|\tfrac{n+2\gamma}{2}}\int_{-L/2}^{L/2} \tilde{v}_L(s+\tau)^2 \,d\tau ds +\int_{-L/2}^{L/2} \tilde{v}_L(\tau)^2\,d\tau \int_{\r \setminus [-\epsilon,\epsilon]}e^{-|s|\tfrac{n+2\gamma}{2}}\, ds<\infty.\\
          %\end{split}
          %\end{equation*}
         % \begin{equation*}
          I_2\sim&\int_{-L/2}^{L/2} \int_{-\epsilon}^{\epsilon} \frac{(\tilde{v}_L(s+\tau)-\tilde{v}_L(\tau))^2}{|s|^{1+2\gamma}} \,ds d\tau
          \sim \int_{-L/2}^{L/2} \int_{ -\epsilon}^{\epsilon} \tilde{v}'_L(\tau)^2|s|^{1-2\gamma}\, ds d\tau<\infty.
          \end{split}
          \end{equation*}
          In this second integral, we have used the Taylor expansion of $\tilde{v}_L$.

On the other hand, $c^{*} (L) = c_{n,\gamma}L^{\frac{{\beta}-1}{{\beta}+1}} \to +\infty$ as $L\to +\infty$. This  proves \eqref{30}.

\end{proof}

\begin{rem}
When $L\rightarrow\infty$, the minimizer $v_L$ for the functional given in \eqref{functional}  satisfies that
$$v_L\rightarrow v_{\infty}\equiv b, $$
where $b(t)$ is defined as in \eqref{b} up to multiplicative constant. The proof of this fact will be postponed to  the forthcoming article \cite{Ao-DelaTorre-Gonzalez-Wei}.
\end{rem}

Let $v$ be a $L$-periodic solution of equation \eqref{equation3}, i.e.,
\begin{equation}\label{eqw}
 \kappa_{n,\gamma}P.V.\int_{0}^L
 (v(t)-v(\tau))K_L(t-\tau)\,d\tau +c_{n,\gamma}v(t)=  c_{n,\gamma}v(t)^{\beta}.
 \end{equation}
The linearization of this equation around the constant solution $v_1\equiv 1$ is:
\begin{equation}\label{linearized}\kappa_{n,\gamma}\int_{0}^L
 (v(t)-v(\tau))K_L(t-\tau)\,d\tau -c_{n,\gamma}(\beta-1) v(t)=0.\end{equation}
We consider the eigenvalue problem for this linearized operator: %
 \begin{equation}\label{vpropio}
 \kappa_{n,\gamma}\int_{0}^L
 (v(t)-v(\tau))K_L(t-\tau)\,d\tau -c_{n,\gamma}(\beta-1)v(t)= \delta_L v(t).
 \end{equation}

\begin{lem}
There exists $\tilde{L}_0>0$ such that
\begin{equation*}
\delta_L<0\text{ if }L>\tilde{L}_0,\quad
\delta_L>0\text{ if }L<\tilde{L}_0,\text{ and }
\delta_{\tilde{L}_0}=0.
\end{equation*}
\end{lem}

\begin{proof}
 Following the computations in \cite{paper1} we get that the first eigenvalue $\delta_L$ is given by the implicit expression
\begin{equation*}
\frac{\left|\Gamma(\tfrac{n}{4}+\tfrac{\gamma}{2}+\tfrac{\sqrt{\lambda}}{2}i)\right|^2}
{\left|\Gamma(\tfrac{n}{4}-\tfrac{\gamma}{2}+\tfrac{\sqrt{\lambda}}{2}i)\right|^2}
=\frac{n+2\gamma}{n-2\gamma}\frac{\left|\Gamma\left(\frac{1}{2}\left(\frac{n}{2}+\gamma\right)\right)\right|^2}
{\left|\Gamma\left(\frac{1}{2}\left(\frac{n}{2}-\gamma\right)\right)\right|^2}+\delta_L.
\end{equation*}
Here $\lambda$ is univocally related with the period by $L=\tfrac{2\pi}{\sqrt{\lambda}}$. %
 $\delta_L$ is a strictly decreasing function of $L$. %
We now define $\tilde{L}_0$  as the period corresponding to the zero eigenvalue.
\end{proof}

We are now ready to conclude the proof of Theorem \ref{th1}. Let
\begin{equation}\label{L0} L_0= \mbox{sup} \{ L \ | \ c(l)=c^{*} (l) \ \mbox{for} \ l \in (0, L)\}.\end{equation}
By Lemma \ref{lemma1} we see that $L_0>0$.  By Lemma \ref{lemma2}, also $ L_0 <+\infty$.  Then we are left to  check that if $L=L_0$ we just have the constant solution.

%\begin{defi}
%Let $v>0$ and $v_1:\equiv1$ be both $L$-periodic solutions of \eqref{eqw}. We define
 %\begin{equation}\label{w}
 %w=v-1.
 %\end{equation}
 %Note that $w>-1$ since $v>0$, and it satisfies \begin{equation}\label{eqw1}
%\kappa_{n,\gamma}P.V.   \int_{0}^L
 %(w(t)-w(\tau))K_L(t-\tau)\,d\tau =c_{n,\gamma}\left[-(w(t)+1)+(w(t)+1)^{\beta}\right],
 %\end{equation}
%which is equivalent to
%\begin{equation*}
%\kappa_{n,\gamma}P.V.   \int_{0}^L
% (w(t)-w(\tau))K_L(t-\tau)\,d\tau - c_{n,\gamma}(\beta-1)w(t) =f(w(t))
% \end{equation*}
 %for
 %\begin{equation*}
 %f(z):=c_{n,\gamma}\left[-(\beta z+1)+(z+1)^{\beta}\right].
 %\end{equation*}
 %Note that if $L\leq\tilde{L}_0$, inequality \eqref{E>0} assures that $f(w)\geq 0$.
%\end{defi}

\begin{prop}\label{L01}
If $L=\tilde{L}_0$ the unique solution for \eqref{eqw} is the constant solution $v_1\equiv 1$.%, i.e, $w\equiv 0$.
\end{prop}

%\begin{proof}
%
%Step 1: If $L\leq\tilde{L}_0$, then the function $w$ defined as in \eqref{w} must satisfy $w\geq 0$. This just follows by sketching the plot of $f(w)\geq 0$.
%
%Step 2: Let $v$ any $L-$periodic smooth solution for \eqref{eqw}, then there exists $t_1\in[0,L]$ where $v$ intersects  the constant solution $v_1\equiv 1$, i.e.,  $w(t_1)=0$.
%In fact, let $a,b\in [0,L]$ satisfy $w(a)=\max_{t\in[0,L]}w(t)$ and $w(b)=\min_{t\in[0,L]}w(t)$, where $w$ is the function defined in \eqref{w}. Then because $w$ satisfies \eqref{eqw1} we have
%$$-w(a)-1+ (w(a)+1)^{\beta}(t)\geq 0$$
%and
%$$-w(b)-1+ (w(b)+1)^{\beta}(t)\leq 0.$$
%Thus we can assert that $w(a)\geq 0$ and $w(b)\in[-1,0]$, and therefore there exists a point $t_1\in[0,L]$ with $w(t_1)=0$.

%Step 3: Applying the maximum principle given in Proposition \ref{strongmaxppo} to equation \eqref{eqw1} we get $w>0$ or $w\equiv 0$. Then step $2$ assures that $w$ is the zero function.

%\end{proof}
\begin{proof}
Let $v>0$ and $v_1\equiv1$ be both $\tilde L_0$-periodic solutions of \eqref{eqw}. We define
 \begin{equation}\label{w}
 w=v-1.
 \end{equation}
On the one hand, using the weak formulation for the problem \eqref{eqw} given in \eqref{formdebil}, we have
\begin{equation}\label{weak1}
    \langle\mathscr{L}_{\gamma}^Lv,\phi\rangle=\langle\mathscr{\tilde L}_{\gamma}^{ L}v,\phi\rangle+c_{n,\gamma}\int_0^Lv(t)\phi(t)\,dt= c_{n,\gamma}\int_0^Lv^{\beta}(t)\phi(t)\,dt,\quad  \forall \phi\in H^{\gamma}_L,
\end{equation}
    where we have defined
    \begin{equation*}
    \langle\mathscr{\tilde L}_{\gamma}^{ L}v,\phi\rangle= \kappa_{n,\gamma}P.V.   \int_{0}^L\int_{0}^L (v(t)-v(\tau))(\phi(t)-\phi
    (\tau))K_L(t-\tau)\,dt\,d\tau.
    \end{equation*}
Thus, in particular for $v=w+1$, equation \eqref{weak1} reads
\begin{equation*}
\langle\mathscr{\tilde L}_{\gamma}^{ L}(w),\phi\rangle+c_{n,\gamma}\int_0^L(w(t)+1)\phi(t)\,dt= c_{n,\gamma}\int_0^L(w(t)+1)^{\beta}\phi(t)\,dt,\quad  \forall \phi\in H^{\gamma}_L,
\end{equation*}
which, interchanging $\phi$ and $w$ in the first term, is equivalent to
\begin{equation}\label{weak2}
\langle\mathscr{\tilde L}_{\gamma}^{ L}(\phi),w\rangle+c_{n,\gamma}\int_0^L\left((w(t)+1)-(w(t)+1)^{\beta}\right)\phi(t)\,dt=0,\quad  \forall \phi\in H^{\gamma}_L.
\end{equation}
On the other hand, if $\varphi_1$ denotes the first eigenfunction for the linearized problem around $v\equiv 1$, given in \eqref{vpropio}, for the period $\tilde{L}_0$ (i.e. the corresponding to the zero eigenvalue $\delta_{\tilde{L}_0}=0$), the following holds
\begin{equation*}
    %\langle\mathscr{L}_{\gamma}^Lv,\phi\rangle=
    \langle\mathscr{\tilde L}_{\gamma}^{ L}\varphi_1,\phi\rangle+c_{n,\gamma}\int_0^L\varphi_1(t)\phi(t)\,dt= {\beta}c_{n,\gamma}\int_0^L\varphi_1(t)\phi(t)\,dt,\quad  \forall \phi\in H^{\gamma}_L.
\end{equation*}
Now we choose the test function here to be $\phi=w$, the function defined in \eqref{w}, and the equality above becomes
\begin{equation}\label{weak3}
    \langle\mathscr{\tilde L}_{\gamma}^{ L}\varphi_1,w\rangle = (\beta-1)c_{n,\gamma}\int_0^L\varphi_1(t)w(t)\,dt.
\end{equation}
Coming back to equation \eqref{weak2} for the test function $\phi=\varphi_1$, then we have
\begin{equation*}
\langle\mathscr{\tilde L}_{\gamma}^{ L}(\varphi_1),w\rangle+c_{n,\gamma}\int_0^L\left((w(t)+1)-(w(t)+1)^{\beta}\right)\varphi_1(t)\,dt=0,
\end{equation*}
which using equality \eqref{weak3} reads
\begin{equation}\label{weak4}
%\int_0^L\left((\beta-1)w(t)+(w(t)+1)-(w(t)+1)^{\beta}\right)\varphi_1(t)\,dt=0,
\int_0^L\left(\beta w(t)+1-(w(t)+1)^{\beta}\right)\varphi_1(t)\,dt=0.
\end{equation}
The positivity of the first eigenfunction $\varphi_1$ and the convexity of the function $f(w)=\beta w(t)+1-(w(t)+1)^{\beta}$ assure that the only possible solution for \eqref{weak4} is $w\equiv 0$.
\end{proof}

%\begin{lem}\label{Energ}
Let $v\in H_L^{\gamma}$ and $E_L$, $\tilde{E}_L$ be the energy functionals for the non-linear and the linear problems \eqref{eqw} and \eqref{linearized} defined by
\begin{equation}\label{E}
E_L(v):=\tfrac{\kappa_{n,\gamma}}{2}\int_0^L\int_{0}^L
 (v(t)-v(\tau))^2K_L(t-\tau)\,d\tau\,dt +\tfrac{c_{n,\gamma}}{2}\int_0^L v^2(t)-\tfrac{c_{n,\gamma}}{\beta+1}\int_0^L v^{\beta+1}(t)\,dt, \end{equation}
 and
\begin{equation}\label{E_lin}
\tilde{E}_L(v):=\tfrac{\kappa_{n,\gamma}}{2}\int_0^L\int_{0}^L
 (v(t)-v(\tau))^2K_L(t-\tau)\,d\tau\,dt -\tfrac{c_{n,\gamma}}{2}(\beta-1)\int_0^L v^2(t)\,dt,\end{equation}
respectively. % then,
%\begin{equation}\label{E>0}
%E_L(v)\geq 0 \text{ for all }L\leq \tilde{L}_0.
%\end{equation}
%\end{lem}
%\begin{proof}
%Indeed, from
The variational formulation of the first eigenvalue $\delta_L$ (Rayleygh quotient) for \eqref{E_lin} implies the following Poincar{\'e} inequality %holds
\begin{equation*}%\label{vp}
%\begin{split}
\kappa_{n,\gamma}\int_0^L\int_{0}^L
 (v(t)-v(\tau))^2K_L(t-\tau)\,d\tau\,dt -c_{n,\gamma}(\beta-1)\int_0^L v^2(t)\,dt
 \geq\delta_L\int_0^L v^2(t)\,dt, \quad \forall v\in H^\gamma_L.
 %\end{split}
\end{equation*}
%\end{proof}
In particular, if $\varphi_1$ denotes, as before, the first eigenfunction for the linearized problem around $v\equiv 1$ at the period $\tilde L_0$, we have the equality
\begin{equation}\label{vp0}
%\begin{split}
\kappa_{n,\gamma}\int_0^{\tilde{L}_0}\int_{0}^{\tilde{L}_0}
 (\varphi_1(t)-\varphi_1(\tau))^2K_{\tilde{L}_0}(t-\tau)\,d\tau\,dt -c_{n,\gamma}(\beta-1)\int_0^{\tilde{L}_0} \varphi_1^2(t)\,dt
 =0.
 %\end{split}
\end{equation}

\begin{prop}
The period $L_0$ defined in \eqref{L0} coincides with the period $\tilde{L}_0$ given by the zero eigenvalue in equation \eqref{vpropio}.
\end{prop}
\begin{proof}
First, because of the definition of $L_0$ given in \eqref{L0} we can easily check that $L_0\geq \tilde{L}_0$. Indeed, Proposition \ref{L01} asserts that $c(\tilde{L}_0)=c^*(\tilde{L}_0)$ and Lemma \ref{L21} assures that this is not possible if $\tilde{L}_0>L_0$.

 We now are going to check the opposite inequality. We have defined $\tilde{L}_0$ as the period where the constant solution $v_1\equiv 1$ loses stability. %Solucion es estable si el primer valor propio es positivo, si es negativo es inestable.
This is, if we define
\begin{equation}\label{Leps}L_{\epsilon}=\tilde{L}_0+\epsilon\end{equation} with $\epsilon>0$, we have instability for the constant solution and thus $c(L_{\epsilon})<c^*(L_{\epsilon})$. %where $c(L)$ and $c^*(L)$ are defined as in \eqref{cL} and \eqref{c*}.
 To prove this, we compute the energy \eqref{E} for the function $1+\sigma\phi_{L_{\epsilon}}$, where $\sigma>0$ small enough and $\phi_{L_{\epsilon}}\in H^{\gamma}_{L_{\epsilon}}$.
 We have
\begin{equation*}
\begin{split}
E_{L_{\epsilon}}(1+\sigma\phi_{L_{\epsilon}})=&E_{L_{\epsilon}}(1)\\
&+\sigma^2\left[\kappa_{n,\gamma}\int_0^{L_{\epsilon}}\int_{0}^{L_{\epsilon}}
 (\phi_{L_{\epsilon}}(t)-\phi_{L_{\epsilon}}(\tau))^2K_{L_{\epsilon}}(t-\tau)\,d\tau\,dt -c_{n,\gamma}(\beta-1)\int_0^{L_{\epsilon}} \phi_{L_{\epsilon}}^2(t)\,dt
\right]\\
&+\text{h.o.t}.
\end{split}
\end{equation*}
Therefore, if we find $\phi_{L_{\epsilon}}$ such that
\begin{equation}\label{ineq}
\kappa_{n,\gamma}\int_0^{L_{\epsilon}}\int_{0}^{L_{\epsilon}}
 (\phi_{L_{\epsilon}}(t)-\phi_{L_{\epsilon}}(\tau))^2K_{L_{\epsilon}}(t-\tau)\,d\tau\,dt -c_{n,\gamma}(\beta-1)\int_0^{L_{\epsilon}} \phi_{L_{\epsilon}}^2(t)\,dt<0,\end{equation}
  the instability of $v_1\in H^{\gamma}_{L_{\epsilon}}$ is proved.  Let $\phi_{L_{\epsilon}}(t)=\varphi_1(\tfrac{\tilde{L_0}}{L_{\epsilon}}t)$. Under the changes of variable $\bar{t}=\tfrac{L_{\epsilon}}{\tilde{L}_0}t$ and $\bar{\tau}=\tfrac{L_{\epsilon}}{\tilde{L}_0}\tau$, equality \eqref{vp0} and Lemma \ref{appendix} imply \eqref{ineq}.
 Here we have also used that $\beta=\tfrac{n+2\gamma}{n-2\gamma}>1$, and $L_{\epsilon}>\tilde{L_0}$ \eqref{Leps}.

 The definition of $L_0$ \eqref{L0} implies $L_0\leq \tilde{L}_0+\epsilon$. Taking limit as $\epsilon$ goes to zero we have the claimed equality $L_0=\tilde{L}_0.$
\end{proof}
This proves completes the proof of Theorem \ref{th1}.
\qed

\bigskip\textbf{Acknowledgements.} The authors would like to thank the anonymous referee for his/her many suggestions and improvements of the paper.

\bibliographystyle{abbrv}

\end{document}